\RequirePackage{fix-cm}
\documentclass[authoryear,smallextended]{svjour3}       

\usepackage{graphicx}

\usepackage{amsmath}
\usepackage{amsfonts}
\usepackage{amssymb}
\usepackage{verbatim}
\usepackage{enumerate}
\usepackage[colorlinks,citecolor=blue,urlcolor=blue]{hyperref}
\usepackage[justification=centering]{caption}
\smartqed

\usepackage[12pt]{extsizes}

\usepackage{booktabs,caption}
\usepackage[flushleft]{threeparttable}

\newenvironment{proofsect}[1]{\vskip0.1cm\noindent{\rmfamily\itshape #1.}}{\qed\vspace{0.15cm}}

\spnewtheorem*{Main Theorem}{Main Theorem}{\normalfont\bfseries}{\itshape}
\spnewtheorem{mylemma}[theorem]{Lemma}{\bfseries}{\itshape} 
\spnewtheorem{myproposition}[theorem]{Proposition}{\bfseries}{\itshape} 
\spnewtheorem{mycorollary}[theorem]{Corollary}{\bfseries}{\itshape} 
\spnewtheorem{mydefinition}[theorem]{Definition}{\bfseries}{\itshape} 
\spnewtheorem{myquestion}{Question}{\bfseries}{\itshape} 
\spnewtheorem{myconjecture}[myquestion]{Conjecture}{\bfseries}{\itshape} 
\numberwithin{equation}{section} \numberwithin{theorem}{section}
\usepackage[utf8]{inputenc}
\usepackage{fourier} 
\usepackage{array}
\usepackage{makecell}
\usepackage{bm}

\usepackage{amsthm}

\makeatletter
\g@addto@macro{\definition}{\itshape}
\makeatother

\usepackage{geometry}
\usepackage{natbib} 
\usepackage{filecontents}

\begin{document}

\title{Large Degree Asymptotics and the Reconstruction Threshold of the Asymmetric Binary Channels}

\titlerunning{Large Degree Asymptotics and the Reconstruction Threshold}        

\author{ Wenjian Liu \and Ning Ning
}


\institute{        
    Wenjian Liu \at
    Dept.of Mathematics and Computer Science,
    Queensborough Community College, City University of New York\\
    Research supported by CUNY Community College Research Grant \#1541\\
    \email{wjliu@qcc.cuny.edu}           
	\and
	Ning Ning (Corresponding Author)\at
	Dept. of Applied Mathematics, University
	of Washington, Seattle\\
	\email{ningnin@uw.edu} 	 
}

\maketitle

\begin{abstract}
In this paper, we consider a broadcasting process in which information is propagated from a given root node on a noisy tree network, and answer the question that whether the symbols at the $n$th level of the tree contain non-vanishing information of the root as $n$ goes to infinity. Although the reconstruction problem on the tree has been studied in numerous contexts including information theory, mathematical genetics and statistical physics, the existing literatures with rigorous reconstruction thresholds established are very limited. 
In the remarkable work of Borgs, Chayes, Mossel and Roch (\textit{The Kesten-Stigum reconstruction bound is tight for roughly symmetric binary channels}. FOCS, IEEE Comput. Soc. (2006): 518–530. Berkeley, CA.), the exact threshold for the reconstruction
problem for a binary asymmetric channel on the
$d$-ary tree is establish, provided that the asymmetry is sufficiently
small, which is the first exact reconstruction threshold
obtained in roughly a decade. In this paper, by means of
refined analyses of moment recursion on a weighted version of the
magnetization, and concentration investigations, we rigorously give a complete answer to the question of how small it needs to be 
to establish the tightness of the reconstruction threshold and further determine its asymptotics of large degrees.

\end{abstract}

\section{Introduction}
\label{sec1a}\vspace{-4mm}
\subsection{Broadcasting Process and the Reconstruction Problem}
We consider the following broadcasting process that can be considered as a communication tree network, as a model
for propagation of a genetic property or as a tree-indexed Markov chain. In this paper, we restrict our attention to the regular $d$-ary
trees, which is an infinite rooted tree where every vertex has
exactly $d$ offspring, denoted by $\mathbb{T}=(\mathbb{V}, \mathbb{E}, \rho)$ with nodes
$\mathbb{V}$ edges $\mathbb{E}$ and root $\rho\in \mathbb{V}$. A configuration on $\mathbb{T}$ is an element of $\mathcal{C}^\mathbb{T}$ with $\mathcal{C}$ being a finite characters set, that is, an assignment of a state in $\mathcal{C}$ to each vertex.  The state of the root $\rho$, denoted by $\sigma_\rho$, is
chosen according to some initial distribution $\pi$ on $\mathcal{C}$.  
This symbol is then propagated in the tree according to the probability transition matrix
$\mathbf{M}=(M_{ij})_{i, j \in \mathcal{C}}$, which functions as the noisy communication channel on each
edge. That is, for each
vertex $v$ having $u$ as its parent, the spin at $v$ is defined
according to the probabilities
$$
\mathbf{P}(\sigma_v=j\mid\sigma_u=i)=M_{i j}, \quad i, j \in \mathcal{C}.
$$
The objective model taken into
account is the asymmetric binary channel with the configuration set
$\mathcal{C}=\{1, 2\}$, whose transition matrix is of the form
$$
\mathbf{M}= \frac{1}{2} \left(
\begin{array}{cc}
1+\theta & 1-\theta \\
1-\theta & 1+\theta \\
\end{array}
\right) + \frac{\Delta}{2}\left(
\begin{array}{cc}
-1 & 1 \\
-1 & 1 \\
\end{array}
\right),
$$
where $|\theta|+|\Delta|\leq 1$ and $\Delta$ is used to describe the deviation of $\mathbf{M}$
from the symmetric channel. It is easy to see that $\mathbf{M}$ has two eigenvalues, 1 and $\theta$. Then we pick a state at the root according to the stationary distribution $\pi=(\pi_1, \pi_2)$ of
	$\mathbf{M}$, which is given by
	\begin{equation*}
	\pi_1=\frac{1}{2}-\frac{\Delta}{2(1-\theta)}\quad \text{and}\quad
	\pi_2=\frac{1}{2}+\frac{\Delta}{2(1-\theta)},
	\end{equation*}
	and without loss of generality, it is convenient to assume that
	$\pi_1\geq\pi_2$.

Recall that the classical Ising model, a mathematical model of ferromagnetism in statistical mechanics, consists of discrete variables that represent magnetic dipole moments of atomic spins that can be in one of two states ($-1$ or $+1$). 
Consider a set of lattice sites $\Lambda$, each with a set of adjacent sites (e.g. a graph) forming a lattice, and for each $k\in \Lambda$, there is a discrete variable $\sigma_k \in \{-1,+1\}$ representing the site's spin. The energy of a configuration $\sigma$ is given by the Hamiltonian function
$$
H(\sigma)=-\displaystyle\sum_{\langle i, j\rangle}J_{ij}\sigma_i \sigma_j-\mu\displaystyle\sum_{j}h_j\sigma_j,
$$
where the notation $\langle i, j\rangle$ indicates that sites $i$ and $j$ are the nearest neighbors, $J_{ij}$ denotes the interaction between two adjacent sites $i,j \in \Lambda$ and $h_j$ models the external magnetic field interaction of site $j \in \Lambda$. 
In this literature, the current model corresponds to the general Ising model with external field on the tree. 

The problem of reconstruction is to analyze whether there exists non-vanishing information on the letter transmitted from the
root, given all the symbols received at the vertices of the $n$th generation, as $n$ goes to infinity. 
We define the distance between probability measures in line with \cite{evans2000broadcasting}. Let $v_+$ and $v_-$ be two probability measures on the same space. Set $f=dv_+/dv$ and $g=dv_-/dv$ where $v:=(v_+ + v_-)/2$. Inferring the root spin $\sigma_{\rho}$ from the spin configurations on the finite vertex set is a basic problem of Bayesian hypothesis testing. The total variation distance, defined as $d_{TV}(v_+ , v_-):=\frac{1}{2}\int |f-g|dv$, can be interpreted as the difference between the probabilities of correct and erroneous inferences. 
Denote $\sigma(n)$ as the spins at distance $n$ from the root and $\sigma^i(n)$ as $\sigma(n)$ conditioned on $\sigma_\rho = i$. Then the reconstruction problem can be mathematical formulated as the following:
\begin{definition}
	The reconstruction problem for the infinite tree $\mathbb{T}$ is
	\textup{\textbf{solvable}} if for some $i, j\in\mathcal{C}$,
	$$
	\limsup_{n\to\infty}d_{TV}(\sigma^i(n), \sigma^j(n))>0.
	$$
	When the $\limsup$
	is $0$, we will say that the model has
	\textup{\textbf{non-reconstruction}} on $\mathbb{T}$.
\end{definition}

\subsection{Background and Applications}

The reconstruction problem arises naturally in statistical physics, where the reconstruction threshold is identified as the threshold for extremality of the infinite-volume Gibbs measure with
free boundary conditions (see \cite{georgii2011gibbs}). 
In \cite{berger2005glauber, martinelli2007fast, tetali2012phase}, the
reconstruction bound is found to have a crucial determination effect on the efficiency of the Glauber
dynamics on trees and random graphs. The reconstruction threshold is also believed to play an
important role in a variety of other contexts, including
phylogenetic reconstruction in evolutionary
biology (\cite{mossel2004phase, daskalakis2006optimal, roch2006short}), communication theory in the study of noisy
computation (\cite{evans2000broadcasting}),  clustering problem in the setting of the stochastic block model (\cite{mossel2012stochastic, mossel2013proof, neeman2014non}), and network tomography (\cite{bhamidi2010network}). For detailed explanation on the reconstruction problem in mixing, phylogeny and replicas, we refer to Section 1.3 in \cite{bernussou1977point}. For other applications of reconstruction, we refer to Section 1.4 in \cite{sly2009reconstruction} and Section 1.3 in \cite{liu2018tightness}, as well as the references therein. 

In this paper, we focus on analyzing the tightness of the reconstruction bound on the tree for asymmetric binary channels, which corresponds to the asymmetric Ising model on the tree in statistical physics term. 
Well known, the
reconstruction problem is closely related to $\lambda$, the second
largest eigenvalue in absolute value of the transition probability matrix, which is $\theta$ in the current model under investigation. \cite{kesten1966additional,kesten1967limit} showed that
the reconstruction problem is solvable if $d\lambda^2>1$, which is known as the
Kesten-Stigum bound. However in the case of larger noise, i.e.
$d\lambda^2 < 1$, one may wonder whether reconstruction problem is still solvable, that is collecting and analyzing the whole set of symbols received at the $n$th
generation to retrieve information transmitted from the root.

First consider the symmetric channel. It was shown in~\cite{bleher1995purity} that the
reconstruction problem is solvable if and only if $d\lambda^2>1$ in the binary
model. For all
other models, it was also known and easy to prove that
$d\lambda^2>1$ implies solvability, while proving
non-reconstructibility turned out to be harder. Although coupling
arguments easily yield non-reconstruction, these arguments are
typically not rigorous. A natural approach to establish non-reconstructibility is
to analyze recursions in terms of random variables, each of whose
values is the expectation of the chain at a vertex, given the state
at the leaves of the subtree below it, and the corresponding
probabilities. Although the reconstruction problem on the tree has been studied in numerous contexts, the existing literatures with rigorous reconstruction thresholds established are very limited. 
\cite{sly2009reconstruction} proved the first exact reconstruction threshold in a nonbinary model by establishing the Kesten–Stigum bound for the $3$-state Potts model on regular trees of large degree, and further established that the Kesten–Stigum bound is not tight for the $q$-state Potts model when $q \geq 5$, which confirms much of the picture conjectured earlier by \cite{mezard2006reconstruction}. \cite{liu2018tightness} considered a $2q$-state symmetric model, with two categories of $q$ states in each category, and 3 transition probabilities (the probability to remain in the same state, the probability to change states but remain in the same category, and the probability to change categories) and showed that the Kesten-Stigum reconstruction bound is not tight when $q \geq 4$. 


%

Next let us turn to the existing results regarding the asymmetric channel. \cite{mossel2001reconstruction, mossel2004survey} showed that the Kesten-Stigum
bound is not the bound for reconstruction in the binary asymmetric
model with sufficiently large asymmetry or in the symmetric Potts model with
sufficiently many characters, which shed the light on exploring the
tightness of the Kesten-Stigum bound. Furthermore, Proposition 12 in \cite{mossel2001reconstruction} implies that for any
asymmetric channel, given $d$ and $\pi$, the reconstructibility is monotone
in $|\theta|$, say, there exist the thresholds $\theta^-<0<\theta^+$ 
such that, there is non-reconstruction when $\theta \in(\theta^-,
\theta^+)$, while it is reconstructible when $\theta<\theta^-$ or
$\theta>\theta^+$. Therefore, the Kesten-Stigum bound mentioned above
implies immediately
$$
\theta^+\leq d^{-1/2}\quad \textup{and}\quad \theta^-\geq -d^{-1/2},
$$
but exact thresholds for non-solvability had not been known.
The breakthrough result in \cite{borgs2006kesten} established
the exact threshold for the reconstruction 
problem with the binary asymmetric channel on the
$d$-ary tree, provided that the asymmetry is sufficiently
small, which is the first exact reconstruction threshold
obtained in roughly a decade.
However, this beautiful result only rigorously proved the existence of $\Delta$ to satisfy the reconstruction criterion, does not answer the question that how small the asymmetry needs to be, therefore rigorously estimating the
range of asymmetry to keep Kesten-Stigum bound tight is a natural question, which will be answered
in the next section.

\subsection{Main Results and Proof Sketch}
In this section, we will present a critical condition of the stationary initial distribution $\pi$ to keep the tightness of the Kesten-Stigum bound, by means of refined
recursive equations of vector-valued distributions and concentration
analyses.
Furthermore, when the Kesten-Stigum bound is not tight, we provide a new reconstruction threshold $C_\pi \in (0,1)$ for sufficiently large $d$.
Since $d\theta^2>1$ always implies reconstruction, it
suffices to consider $d\theta^2\leq1$ in the following context.

\begin{theorem}
	\label{reconstruction}
	For every $d$ and $\pi$ such that $\pi_1\pi_2<\frac16$,	the Kesten-Stigum bound is not tight. In other words, the
	reconstruction problem is solvable for some $\theta$, even if
	$d\theta^2<1$.
\end{theorem}
The proof to Theorem \ref{reconstruction} above is given in Section \ref{Proof_of_Theorem_1.1}. The proofs to Theorem \ref{large_degree} and Theorem \ref{nonreconstruction} below are given in Section \ref{Proof_of_Theorem_1.2} and Section \ref{Proof_of_Theorem_1.3} respectively.

\begin{theorem}
	\label{large_degree} For every $\pi$ such that $\pi_1\pi_2<\frac16$, there exists an
	asymptotic result of the reconstruction threshold, that is, when $d$ goes to infinity,
$$\lim_{d\to \infty} d\left(\theta^{\pm}\right)^2=C_\pi,$$
	where $C_\pi$ is a constant taking values in $(0,1)$ and depends only on $\pi$.
\end{theorem}

\begin{theorem}
	\label{nonreconstruction} For every $\pi$ such that $\pi_1\pi_2>\frac16$, there
	exists a $D=D(\pi)>0$, such that for $d>D$ the Kesten-Stigum bound is
	sharp, that is
	$$
	\theta^+=d^{-1/2}\quad\textup{and}\quad \theta^-=-d^{-1/2}.
	$$
	Furthermore, there is non-reconstruction at the Kesten-Stigum bound,
	when $\theta=\theta^+$ or $\theta^-$.
\end{theorem}

The idea to establish Theorem \ref{reconstruction}, Theorem \ref{large_degree}  and Theorem \ref{nonreconstruction} is the following.
One standard way to classify reconstruction and non-reconstruction is to analyze the quantity $x_n$: the probability of giving a correct guess of the root given the spins $\sigma(n)$ at distance $n$ from the root, minus the probability of guessing the root according to stationary initial distribution. Non-reconstruction means that the mutual information between the root and the spins at distance $n$ goes to zero as $n$ tends to infinity. In Lemma \ref{nonreconstruction_equivalent}, we rigorously show that $x_n$ is always positive and the non-reconstruction is equivalent to
$$
\lim_{n\to \infty}x_n=0.
$$

To analyze whether the reconstruction holds, inspired by \cite{chayes1986mean}, \cite{borgs2006kesten} and \cite{sly2009reconstruction}, we establish the distributional recursion and moment recursion, and then the recursive relation between the $n$th and the $(n+1)$th generation's structure of the tree leads to a corresponding nonlinear dynamical system. In the mean time, we show that the interactions between spins become very weak, if they are sufficiently far away from each other. Therefore, under this weak interacting situation, i.e. $x_n$ being sufficiently small, the concentration analysis is successfully developed and an approximation to the dynamical system is nicely established:
\begin{equation*}
\label{X}
x_{n+1}\approx d\theta^2x_n+\frac{1-6\pi_1\pi_2}{\pi_1\pi_2^2}\frac{d(d-1)}{2}\theta^4x_n^2.
\end{equation*}

The sign of coefficient of the quadratic term which is determined by $1-6\pi_1\pi_2$, plays a crucial role in the asymptotic behavior of $x_n$. When $1-6\pi_1\pi_2>0$, equivalently $\Delta^2>(1-\theta)^2/3$, if $d\theta^2$ is sufficiently close to 1, then $x_n$ does not converge to $0$ and then there is reconstruction beyond the Kesten-Stigum bound. Then our focus is to find this new reconstruction threshold, which is executed in the following three steps: Step one, we rigorously show that, when degree $d$ is large, the interactions between spins become very weak; Step two, using the Central Limit Theorem, we approximate the corresponding collection of small independent samples, to show that the reconstruction function can be asymptotically given by a new Gaussian approximation function $g(s)$, that is, $x_{n+1}\approx g(d\theta^2 x_n)$; Step three, we explore the first several major terms of the Maclaurin series of $g(s)$, and rigorously establish the reconstruction threshold by discussing the fixed point of $g(s)$. On the other hand, when $1-6\pi_1\pi_2<0$, the analysis of large degree asymptotics yields $g(s)<s$, which implies $\lim_{n\to \infty}x_n=0$, that is, there is non-reconstruction.

\section{Preliminaries}
\subsection {Notations}
Let $u_1,\ldots,u_d$ be the children of $\rho$ and $\mathbb{T}_v$ be the
subtree of descendants of $v\in \mathbb{T}$. Denote the $n$th level of the tree as $L(n)=\{v\in\mathbb{V}:
d(\rho, v)=n\}$, where $d(\cdot, \cdot)$ is the graph distance on $\mathbb{T}$. With the notations above, let $\sigma(n)$ and
$\sigma_j(n)$ be the spins on $L_n$ and $L(n)\cap
\mathbb{T}_{u_j}$ respectively. For a configuration $A$ on $L(n)$,
define the posterior function $f_n$ by
$$
f_n(i, A)=\mathbf{P}(\sigma_\rho=i\mid\sigma(n)=A).
$$
By the recursive nature of the tree, for a configuration $A$ on
$L(n+1) \cap \mathbb{T}_{u_j}$, an equivalent form is given by
\begin{equation*}
f_n(i, A)=\mathbf{P}(\sigma_{u_j}=i\mid\sigma_j(n+1)=A).
\end{equation*}
Next, with $i=1, 2$, define
\begin{equation*}
X_i=X_i(n)=f_n(i, \sigma(n)),\quad
X^+=X^+(n)=f_n(1, \sigma^1(n)),\quad
X^-=X^-(n)=f_n(2, \sigma^2(n)),
\end{equation*}
and for $1\leq j\leq d$,
\begin{equation*}
Y_j=Y_j(n)=f_n(1, \sigma_j^1(n+1)),
\end{equation*}
where it is clear that the random variables $\{Y_j\}_{1\leq j\leq
	d}$ are independent and identical in distribution. It is
apparent that
\begin{equation}
\label{identity1} X_1(n)+X_2(n)=1
\end{equation}
and
\begin{equation}
\label{stationary} \mathbf{E}(X_1)=\pi_1,\quad
\mathbf{E}(X_2)=\pi_2.
\end{equation}
We introduce the objective quantities in this paper:
$$
x_n=\mathbf{E}(X^+(n)-\pi_1)
\quad\textup{and}\quad
z_n=\mathbf{E}(X^+(n)-\pi_1)^2.
$$

\subsection{Preparations}
Before proceeding to the analysis, it is convenient to firstly
derive some very useful identities concerning $x_n$.
\begin{lemma}
	\label{lemma1}
	For any $n\in \mathbb{N}\cup\{0\}$, we have
	\begin{equation*}
	x_n=\frac{1}{\pi_1}\mathbf{E}(X_1-\pi_1)^2=\mathbf{E}(X^+(n)-\pi_1)^2+\frac{\pi_2}{\pi_1}\mathbf{E}(X^-(n)-\pi_2)^2\geq
	z_n\geq 0.
	\end{equation*}
\end{lemma}
\begin{proof}By Bayes' rule, we have
	\begin{eqnarray*}
		\mathbf{E}X^+&=&\mathbf{E}f_n(1,\sigma^1(n))
		\\
		&=&\sum_A f_n(1,A)\mathbf{P}(\sigma(n)=A\mid\sigma_\rho=1)
		\\
		&=&\frac{1}{\pi_1}\sum_A\mathbf{P}(\sigma_\rho=1\mid\sigma(n)=A)\mathbf{P}(\sigma(n)=A)f_n(1,A)
		\\
		&=&\frac{1}{\pi_1}\sum_Af_n(1,A)^2\mathbf{P}(\sigma(n)=A)
		\\
		&=&\frac{1}{\pi_1}\mathbf{E}(X_1^2)
	\end{eqnarray*}
	and similarly,
	$$
	\mathbf{E}X^-=\mathbf{E}f_n\left(2,\sigma^2(n)\right)=\frac{1}{\pi_2}\mathbf{E}(X_2^2).
	$$
	Then it follows from equation \eqref{stationary} that
	\begin{equation}
	\label{identityxn}
	x_n=\frac{1}{\pi_1}\left(\mathbf{E}(X_1^2)-\pi_1^2\right)=\frac{1}{\pi_1}\mathbf{E}(X_1-\pi_1)^2.
	\end{equation}
	Next by equation \eqref{identity1}, one has
	\begin{eqnarray*}
		x_n=\frac{1}{\pi_1}\mathbf{E}(X_2-\pi_2)^2
		=\frac{\pi_2}{\pi_1}(\mathbf{E}X^-(n)-\pi_2).
	\end{eqnarray*}
	Therefore, the quantitative relation between $x_n$ and $z_n$ holds:
	\begin{eqnarray*}
		x_n&=&\frac{1}{\pi_1}\mathbf{E}(X_1-\pi_1)^2
		\\
		&=&\frac{1}{\pi_1}\left[\mathbf{P}(\sigma_\rho=1)\mathbf{E}\big((X_1-\pi_1)^2\mid\sigma_\rho=1\big)+\mathbf{P}(\sigma_\rho=2)\mathbf{E}\big((X_2-\pi_2)^2\mid\sigma_\rho=2\big)\right]
		\\
		&=&\frac{1}{\pi_1}\left[\pi_1\mathbf{E}(X^+(n)-\pi_1)^2+\pi_2\mathbf{E}(X^-(n)-\pi_2)^2\right]
		\\
		&=&\mathbf{E}(X^+(n)-\pi_1)^2+\frac{\pi_2}{\pi_1}\mathbf{E}(X^-(n)-\pi_2)^2
		\\
		&\geq&z_n
	\end{eqnarray*}
	
\end{proof}

With the preceding results, we calculate the means and
variances of $Y_j$.
\begin{lemma}
	\label{covariance}
	For each $1\leq j\leq d$, we have
	$$
	\mathbf{E}(Y_j-\pi_1)=\theta x_n
\quad\textup{and}\quad
	\mathbf{E}(Y_j-\pi_1)^2=\pi_1x_n+\theta(z_n-\pi_1x_n).
	$$
\end{lemma}
\begin{proof}
	If $\sigma_{u_j}^1=1$, $Y_j$ is distributed according to $X^+(n)$,
	while if $\sigma_{u_j}^1=2$, $Y_j$ is distributed according to $1-X^-(n)$. Therefore we have
$$
		\mathbf{E}(Y_j-\pi_1)=\mathbf{P}(\sigma_{u_j}^1=1)\mathbf{E}(X^+(n)-\pi_1)+\mathbf{P}(\sigma_{u_j}^1=2)\mathbf{E}(1-X^-(n)-\pi_1)
	=M_{11}x_n-M_{12}\frac{\pi_1}{\pi_2}x_n
	=\theta x_n
$$
	and similarly we have
	\begin{eqnarray*}
		\mathbf{E}(Y_j-\pi_1)^2&=&\mathbf{P}(\sigma_{u_j}^1=1)\mathbf{E}(X^+(n)-\pi_1)^2+\mathbf{P}(\sigma_{u_j}^1=2)\mathbf{E}(1-X^-(n)-\pi_1)^2
		\\
		&=&M_{11}\mathbf{E}(X^+(n)-\pi_1)^2+M_{12}\mathbf{E}(X^-(n)-\pi_2)^2
		\\
		&=&M_{11}z_n+M_{12}\frac{\pi_1}{\pi_2}(x_n-z_n)
		\\
		&=&\pi_1x_n+\theta(z_n-\pi_1x_n),
	\end{eqnarray*}
	as desired.
\end{proof}

\subsection{An Equivalent Condition for Non-reconstruction}
\label{sec3}
\indent If the reconstruction problem is solvable,
then $\sigma(n)$ contains significant information on the root
variable, which may be formulated in several equivalent ways (see \cite{mossel2001reconstruction}, Proposition 14). As a result, in order to
analyze the reconstruction, it suffices to investigate the
asymptotic behavior of $x_n$ as $n$ goes to infinity.
\begin{lemma}
	\label{nonreconstruction_equivalent}
	The non-reconstruction is equivalent to
	$$
	\lim_{n\to\infty}x_n=0.
	$$
\end{lemma}
\begin{proof}
The maximum-likelihood algorithm, which is the optimal
reconstruction algorithm of $\sigma_\rho$ given $\sigma(n)$, is
successful with probability
\begin{equation*}
\Delta_n=\mathbf{E}\max\{X_1(n), X_2(n)\}.
\end{equation*}
Therefore, the inequality of $x_n+\pi_1\leq
\Delta_n$ holds, which is an analogous result to that of \cite{mezard2006reconstruction}, by noting that the
algorithm that chooses $\sigma_{\rho}$ randomly according to probabilities $X_i$ is correct with probability $x_n+\pi_1$. On the other hand, recalling the assumption that
	$\pi_1\geq\pi_2$, by the Cauchy-Schwartz inequality together with the identities equation \eqref{identity1} and equation \eqref{identityxn},
	one can conclude
	\begin{equation}
	\label{MLE}
	\begin{aligned}
	\Delta_n&=\pi_1+\mathbf{E}\max\left\{X_1(n)-\pi_1,
	X_2(n)-\pi_1\right\}
	\\
	&\leq\pi_1+\mathbf{E}\max\left\{X_1(n)-\pi_1,
	X_2(n)-\pi_2\right\}
	\\
	&=\pi_1+\mathbf{E}|X_1(n)-\pi_1|
	\\
	&\leq\pi_1+\left(\mathbf{E}(X_1(n)-\pi_1)^2\right)^{1/2}
	\\
	&\leq\pi_1+\pi_1^{1/2}x_n^{1/2}.
	\end{aligned}
	\end{equation}
	Hence, one has
	\begin{eqnarray*}
		x_n\leq\Delta_n-\pi_1\leq\pi_1^{1/2}x_n^{1/2},
	\end{eqnarray*}
	implying that $x_n$ converging to $0$ is equivalent to
	$\Delta_n$ converging to $\pi_1$, which is equivalent to
	non-reconstruction (see \cite{ mossel2001reconstruction}).
\end{proof}

\section{Moment Recursion}
\subsection{Distributional Recursion}

In the last section, it is known that the asymptotic behavior of $x_n$
as $n$ goes to infinity plays a crucial role, however it is still too difficult and not necessary
to get the explicit expression for $x_n$. In fact, we only need to
investigate the recursive formula of $x_n$, from which it is
possible to illustrate the trend of $x_n$ as $n$ goes to infinity.
Thus the key idea is to analyze the recursive relation between
$X^+(n)$ and $X^+(n+1)$ by the structure of the tree. Suppose that $A$ is
a configuration on $L(n+1)$ and let $A_j$ be its restriction on
$\mathbb{T}_{u_j}\cap L(n+1)$. Then from the Markov random field
property, we have
\begin{equation}
\label{recursion}
\begin{aligned}
f_{n+1}(1,A)&=\frac{N_1}{N_1+N_2}
\\
&=\frac{\pi_1\prod_{j=1}^d\left[\frac{M_{11}}{\pi_1}f_n(1,A_j)+\frac{M_{12}}{\pi_2}f_n(2,A_j)\right]}{\pi_1\prod_{j=1}^d\left[\frac{M_{11}}{\pi_1}f_n(1,A_j)+\frac{M_{12}}{\pi_2}f_n(2,A_j)\right]+
	\pi_2\prod_{j=1}^d\left[\frac{M_{21}}{\pi_1}f_n(1,A_j)+\frac{M_{22}}{\pi_2}f_n(2,A_j)\right]}
\\
&=\frac{\pi_1\prod_{j=1}^d\left[1+\frac{\theta}{\pi_1}(f_n(1,A_j)-\pi_1)\right]}{\pi_1\prod_{j=1}^d\left[1+\frac{\theta}{\pi_1}(f_n(1,A_j)-\pi_1)\right]+
	\pi_2\prod_{j=1}^d\left[1-\frac{\theta}{\pi_2}(f_n(1,A_j)-\pi_1)\right]},
\end{aligned}
\end{equation}
where
\begin{eqnarray*}
	N_1&=&\pi_1\prod_{j=1}^d[M_{11}\mathbf{P}(\sigma_j(n+1)=A_j\mid\sigma_{u_j}=1)+M_{12}\mathbf{P}(\sigma_j(n+1)=A_j\mid\sigma_{u_j}=2)]
	\\
	&=&\pi_1\prod_{j=1}^d\left[\frac{M_{11}}{\pi_1}f_n(1,A_j)+\frac{M_{12}}{\pi_2}f_n(2,A_j)\right]\mathbf{P}(\sigma_j(n+1)=A_j)
\end{eqnarray*}
and
\begin{eqnarray*}
	N_2&=&\pi_2\prod_{j=1}^d\left[M_{21}\mathbf{P}(\sigma_j(n+1)=A_j\mid\sigma_{u_j}=1)+M_{22}\mathbf{P}(\sigma_j(n+1)=A_j\mid\sigma_{u_j}=2)\right]
	\\
	&=&\pi_2\prod_{j=1}^d\left[\frac{M_{21}}{\pi_1}f_n(1,A_j)+\frac{M_{22}}{\pi_2}f_n(2,A_j)\right]\mathbf{P}(\sigma_j(n+1)=A_j).
\end{eqnarray*}

Next conditioning the root being $1$ and setting $A=\sigma^1(n+1)$, we have
$$
X^+(n+1)=\frac{\pi_1Z_1}{\pi_1Z_1+\pi_2Z_2},
$$
where
$$
Z_1=\prod_{j=1}^d\left[1+\frac{\theta}{\pi_1}(f_n(1,A_j)-\pi_1)\right]=\prod_{j=1}^d\left[1+\frac{\theta}{\pi_1}(Y_j(n)-\pi_1)\right]
$$
and
$$
Z_2=\prod_{j=1}^d\left[1-\frac{\theta}{\pi_2}(f_n(1,A_j)-\pi_1)\right]=\prod_{j=1}^d\left[1-\frac{\theta}{\pi_2}(Y_j(n)-\pi_1)\right].
$$

\subsection{Main Expansion of $x_{n+1}$}

With the help of those preliminary results, we are about to figure out the
recursive relation regarding $x_{n+1}$, specifically, its main
expansion, which would play a crucial rule in further
discussions. Firstly we take care of the approximating means and variances of
$Z_i$ and the Taylor
series approximations.
\begin{lemma} \label{Taylor} For each positive integer $k$, there
	exists a $C=C(\pi, k)$ which only depends on $\pi$ and $k$, such that for
	each $0\leq \ell, m\leq k$,
	$$
	\mathbf{E}Z_1^{\ell}Z_2^{m}\leq C,\quad
	\left|\mathbf{E}Z_1^{\ell}Z_2^{m}-1-du\right|\leq
	Cx_n^2,\quad
		\left|\mathbf{E}Z_1^{\ell}Z_2^{m}-1-du-\frac{d(d-1)}{2}u^2\right|\leq
		Cx_n^3,
$$
where $$u=\mathbf{E}\left[1+\frac{\theta}{\pi_1}(Y_1(n)-\pi_1)\right]^{\ell}\left[1-\frac{\theta}{\pi_2}(Y_1(n)-\pi_1)\right]^{m}-1.$$
\end{lemma}

\begin{proof}
Since $\{Y_j\}_{1\leq j\leq
	d}$ are independent and identical in distribution, we have
        \begin{equation*}
		\begin{aligned}
		\mathbf{E}Z_1^{\ell}Z_2^{m}&=\prod_{j=1}^d\mathbf{E}\left[1+\frac{\theta}{\pi_1}(Y_j(n)-\pi_1)\right]^{\ell}\left[1-\frac{\theta}{\pi_2}(Y_j(n)-\pi_1)\right]^{m}
		\\
		&=\left(\mathbf{E}\left[1+\frac{\theta}{\pi_1}(Y_1(n)-\pi_1)\right]^{\ell}\left[1-\frac{\theta}{\pi_2}(Y_1(n)-\pi_1)\right]^{m}\right)^d=(1+u)^d.
		\end{aligned}
		\end{equation*}
It follows from $0\leq Y_{1}\leq 1$ that $|Y_1(n)-\pi_1|\leq1$, and then when $i\geq2$, we have $|Y_1(n)-\pi_1|^i\leq(Y_1(n)-\pi_1)^2$. Therefore Lemma~\ref{covariance} implies that
\begin{equation*}
\begin{aligned}
|u|&=\left|\mathbf{E}\left[1+\frac{\theta}{\pi_1}(Y_1(n)-\pi_1)\right]^{\ell}\left[1-\frac{\theta}{\pi_2}(Y_1(n)-\pi_1)\right]^{m}-1\right|
\\
&=\left|\sum_{i=1}^{\ell+m}c_i\mathbf{E}[\theta(Y_1-\pi_1)]^i\right|
\\
&\leq |c_1|\theta^2x_n+\sum_{i=2}^{\ell+m}|c_i|\theta^2\left[\pi_1x_n+\theta(z_n-\pi_1x_n)\right]
\\
&\leq c\theta^2x_n,
\end{aligned}
\end{equation*}
where $\{c_i\}_{i=1}^{l+m}$ and $c$ depend on $\pi$ and $k$	only. Consequently, we have $d|u|\leq cx_n$ by means of $d\theta^2\leq1$. Using the
	binomial expansion and the Remainder Theorem, we have
	\begin{eqnarray*}
		\left|(1+u)^d-\sum_{i=0}^h{d \choose i} u^i \right| \leq
		\sum_{i=h+1}^\infty\frac{d^i}{i!}|u|^i\leq e^{c}c^{h+1}x_n^{h+1}.
	\end{eqnarray*}
Taking $h=0, 1, 2$ respectively and $\displaystyle C=\max_{h\in \{0, 1, 2\}}\left\{e^{c}c^{h+1}\right\}$ complete the proof.
\end{proof}

Next we aim to figure out the recursive relation of $x_{n+1}$ by virtue of the following identity
\begin{equation}
\label{identity}
\frac{a}{s+r}=\frac{a}{s}-\frac{ar}{s^2}+\frac{r^2}{s^2}\frac{a}{s+r}.
\end{equation}
Specifically, plugging $a=\pi_1Z_1$, $r=\pi_1Z_1+\pi_2Z_2-1$ and
$s=1$ in equation \eqref{identity} yields
\begin{equation}
\label{expansion}
\begin{aligned} x_{n+1}&=\mathbf{E}X^+(n+1)-\pi_1
\\
&=\mathbf{E}(\pi_1Z_1)-\mathbf{E}[\pi_1Z_1(\pi_1Z_1+\pi_2Z_2-1)]
+\mathbf{E}\left[(\pi_1Z_1+\pi_2Z_2-1)^2\frac{\pi_1Z_1}{\pi_1Z_1+\pi_2Z_2}\right]-\pi_1.
\end{aligned}
\end{equation}
In the following, we analyze terms
in equation \eqref{expansion}, using
the notation $O_\pi$ to emphasize that the constant associated with the $O$-term depends on $\pi$ only
\begin{eqnarray*}
	\mathbf{E}Z_1&=&\mathbf{E}\prod_{j=1}^d\left[1+\frac{\theta}{\pi_1}(Y_j-\pi_1)\right]
	\\
	&=&1+\frac{d\theta}{\pi_1}\mathbf{E}(Y_1-\pi_1)+\frac{d(d-1)}{2}\left[\frac{\theta}{\pi_1}\mathbf{E}(Y_1-\pi_1)\right]^2+O_\pi(x_n^3)
	\\
	&=&1+\frac{d\theta^2}{\pi_1}x_n+\frac{d(d-1)}{2}\frac{\theta^4}{\pi_1^2}x_n^2+O_\pi(x_n^3),
\end{eqnarray*}
\begin{eqnarray*}
	\mathbf{E}Z_1^2&=&\mathbf{E}\prod_{j=1}^d\left[1+\frac{\theta}{\pi_1}(Y_j-\pi_1)\right]^2
	\\
	&=&1+d\left\{\mathbf{E}\left[1+\frac{\theta}{\pi_1}(Y_1-\pi_1)\right]^2-1\right\}+\frac{d(d-1)}{2}\left\{\mathbf{E}\left[1+\frac{\theta}{\pi_1}(Y_1-\pi_1)\right]^2-1\right\}^2+O_\pi(x_n^3)
	\\
	&=&1+d\left[\frac{\theta^2}{\pi_1}(3-\theta)x_n+\frac{\theta^3}{\pi_1^2}z_n\right]+\frac{d(d-1)}{2}\left[\frac{\theta^2}{\pi_1}(3-\theta)x_n+\frac{\theta^3}{\pi_1^2}z_n\right]^2+O_\pi(x_n^3),
\end{eqnarray*}
and
\begin{equation*}
\begin{aligned}
\mathbf{E}Z_1Z_2
&=\mathbf{E}\prod_{j=1}^d\left[1+\frac{\theta}{\pi_1}(Y_j-\pi_1)\right]\left[1-\frac{\theta}{\pi_2}(Y_j-\pi_1)\right]
\\
&=1+d\left\{\mathbf{E}\left[1+\frac{\theta}{\pi_1}(Y_1-\pi_1)\right]\left[1-\frac{\theta}{\pi_2}(Y_1-\pi_1)\right]-1\right\}
\\
&\quad+\frac{d(d-1)}{2}\left\{\mathbf{E}\left[1+\frac{\theta}{\pi_1}(Y_1-\pi_1)\right]\left[1-\frac{\theta}{\pi_2}(Y_1-\pi_1)\right]-1\right\}^2+O_\pi(x_n^3)
\\
&=1+d\left[\theta^2\left(\frac{1}{\pi_1}+\frac{\theta-2}{\pi_2}\right)x_n-\frac{\theta^3}{\pi_1\pi_2}z_n\right]
+\frac{d(d-1)}{2}\left[\theta^2\left(\frac{1}{\pi_1}+\frac{\theta-2}{\pi_2}\right)x_n-\frac{\theta^3}{\pi_1\pi_2}z_n\right]^2+O_\pi(x_n^3).
\end{aligned}
\end{equation*}
Then the preceding results yield
\begin{equation*}
\begin{aligned}
&\mathbf{E}\pi_1Z_1(\pi_1Z_1+\pi_2Z_2-1)
\\
&=\pi_1^2\mathbf{E}Z_1^2+\pi_1\pi_2\mathbf{E}Z_1Z_2-\pi_1\mathbf{E}Z_1
\\
&=\pi_1^2\frac{d(d-1)}{2}\left[\frac{\theta^2}{\pi_1}(3-\theta)x_n+
\frac{\theta^3}{\pi_1^2}z_n\right]^2+\pi_1\pi_2\frac{d(d-1)}{2}\left[\theta^2\left(\frac{1}{\pi_1}+\frac{\theta-2}{\pi_2}\right)x_n-\frac{\theta^3}{\pi_1\pi_2}z_n\right]^2
\\
&\quad-\frac{d(d-1)}{2}\frac{\theta^4}{\pi_1}x_n^2+O_\pi(x_n^3).
\end{aligned}
\end{equation*}
Similarly, we have
\begin{eqnarray*}
	\mathbf{E}Z_2&=&\mathbf{E}\prod_{j=1}^d\left[1-\frac{\theta}{\pi_2}(Y_j-\pi_1)\right]
	\\
	&=&1-\frac{d\theta}{\pi_2}\mathbf{E}(Y_1-\pi_1)+\frac{d(d-1)}{2}\left[\frac{\theta}{\pi_2}\mathbf{E}(Y_1-\pi_1)\right]^2+O_\pi(x_n^3)
	\\
	&=&1-\frac{d\theta^2}{\pi_2}x_n+\frac{d(d-1)}{2}\frac{\theta^4}{\pi_2^2}x_n^2+O_\pi(x_n^3),
\end{eqnarray*}
\begin{eqnarray*}
	\mathbf{E}Z_2^2&=&\mathbf{E}\prod_{j=1}^d\left[1-\frac{\theta}{\pi_2}(Y_j-\pi_1)\right]^2
	\\
	&=&1+d\left\{\mathbf{E}\left[1-\frac{\theta}{\pi_2}(Y_1-\pi_1)\right]^2-1\right\}+\frac{d(d-1)}{2}\left\{\mathbf{E}\left[1-\frac{\theta}{\pi_2}(Y_1-\pi_1)\right]^2-1\right\}^2+O_\pi(x_n^3)
	\\
	&=&1+d\left\{\frac{\theta^2}{\pi_2}\left[\frac{\pi_1}{\pi_2}(1-\theta)-2\right]x_n+\frac{\theta^3}{\pi_2^2}z_n\right\}
	+\frac{d(d-1)}{2}\left\{\frac{\theta^2}{\pi_2}\left[\frac{\pi_1}{\pi_2}(1-\theta)-2\right]x_n+\frac{\theta^3}{\pi_2^2}z_n\right\}^2+O_\pi(x_n^3),
\end{eqnarray*}
and then
\begin{equation}
\label{square}
\begin{aligned}
&\mathbf{E}(\pi_1Z_1+\pi_2Z_2-1)^2
\\
=&\pi_1^2\mathbf{E}(Z_1^2)+\pi_2^2\mathbf{E}(Z_2^2)+2\pi_1\pi_2\mathbf{E}(Z_1Z_2)-2\pi_1\mathbf{E}(Z_1)-2\pi_2\mathbf{E}(Z_2)+1
\\
=&\pi_1^2\frac{d(d-1)}{2}\left[\frac{\theta^2}{\pi_1}(3-\theta)x_n+\frac{\theta^3}{\pi_1^2}z_n\right]^2
+\pi_2^2\frac{d(d-1)}{2}\left\{\frac{\theta^2}{\pi_2}\left[\frac{\pi_1}{\pi_2}(1-\theta)-2\right]x_n+\frac{\theta^3}{\pi_2^2}z_n\right\}^2
\\
&\quad+\pi_1\pi_2d(d-1)\left[\theta^2\left(\frac{1}{\pi_1}+\frac{\theta-2}{\pi_2}\right)x_n-\frac{\theta^3}{\pi_1\pi_2}z_n\right]^2
-d(d-1)\frac{\theta^4}{\pi_1}x_n^2-d(d-1)\frac{\theta^4}{\pi_2}x_n^2+O_\pi(x_n^3).
\end{aligned}
\end{equation}
As a consequence, we have
\begin{equation}
\label{explicit}
\begin{aligned}
x_{n+1}&=\mathbf{E}(\pi_1Z_1)-\mathbf{E}\pi_1Z_1(\pi_1Z_1+\pi_2Z_2-1)+\pi_1\mathbf{E}(\pi_1Z_1+\pi_2Z_2-1)^2-\pi_1+S
\\
&=d\theta^2x_n+\frac{d(d-1)}{2}\theta^4x_n^2\left\{\left(\frac{2}{\pi_1}-\frac{2}{\pi_2}\right)-\pi_2\left[3+\frac{\theta}{\pi_1}\left(\frac{z_n}{x_n}-\pi_1\right)\right]^2+\pi_1\pi_2(\pi_1-\pi_2)\right.
\\
&\quad\left.\times\left[\left(\frac{1}{\pi_1}-\frac{2}{\pi_2}\right)-\frac{\theta}{\pi_1\pi_2}\left(\frac{z_n}{x_n}-\pi_1\right)\right]^2
+\pi_1\left[\frac{\pi_1}{\pi_2}-2+\frac{\theta}{\pi_2}\left(\frac{z_n}{x_n}-\pi_1\right)\right]^2\right\}+S+O_\pi(x_n^3)
\\
&=d\theta^2x_n+\frac{1-6\pi_1\pi_2}{\pi_1\pi_2^2}\frac{d(d-1)}{2}\theta^4x_n^2+R+S,
\end{aligned}
\end{equation}
where
\begin{equation}
\label{R} |R|\leq
C_Rx_n^2\left(\frac{d(d-1)}{2}|\theta|^5\left|\frac{z_n}{x_n}-\pi_1\right|+x_n\right)
\end{equation}
with $C_R$ a constant depending only on $\pi$, and
\begin{equation}
\label{sly2009reconstruction}
S=\mathbf{E}(\pi_1Z_1+\pi_2Z_2-1)^2\left(\frac{\pi_1Z_1}{\pi_1Z_1+\pi_2Z_2}-\pi_1\right)
\end{equation}
which will be handled in the following concentration investigation.

\section{Concentration Analysis}
Noting that $Z_1, Z_2\geq0$, we have $0\leq
\frac{\pi_1Z_1}{\pi_1Z_1+\pi_2Z_2}\leq 1$. It is concluded
from equations \eqref{expansion} and \eqref{square} that
\begin{equation}
\label{lemLA}
|x_{n+1}-d\theta^2x_n|\leq Cx_n^2\leq\varepsilon x_n,
\end{equation}
where $C=C(\pi)$ depends only on $\pi$. In equation \eqref{lemLA}, the first inequality follows
from Lemma~\ref{lemma1} which states that $0\leq z_n\leq x_n$, and the last inequality holds if
$x_n<\delta$ for $\delta=\delta(\pi, \varepsilon)$ small enough.  
The following lemma ensures that $x_n$ does not drop too fast.
\begin{lemma}
	\label{ndtf} For any $\varrho>0$, there exists a constant
	$\gamma=\gamma(\pi, \varrho)>0$, such that for all $n$ when
	$|\theta|>\varrho$,
	$$
	x_{n+1}\geq \gamma x_n.
	$$
\end{lemma}
\begin{proof} For a configuration $A$ on $\mathbb{T}_{u_1}\cap L(n+1)$, we have
	\begin{eqnarray*}
		f_{n+1}^*(1, A)&:=&\mathbf{P}(\sigma_\rho=1\mid\sigma_1(n+1)=A)
		\\
		&=&\pi_1\frac{\mathbf{P}(\sigma_1(n+1)=A\mid\sigma_\rho=1)}{\mathbf{P}(\sigma_1(n+1)=A)}
		\\
		&=&\pi_1\frac{M_{11}\mathbf{P}(\sigma_1(n+1)=A\mid\sigma_{u_1}=1)+M_{12}\mathbf{P}(\sigma_1(n+1)=A\mid\sigma_{u_1}=2)}{\mathbf{P}(\sigma_1(n+1)=A)}
		\\
		&=&\pi_1\left[\frac{M_{11}}{\pi_1}f_n(1,
		A)+\frac{M_{12}}{\pi_2}\left(1-f_n(1, A)\right)\right]
		\\
		&=&\pi_1\left[1+\frac{\theta}{\pi_1}(f_n(1, A)-\pi_1)\right],
	\end{eqnarray*}
	and then
	$$
	\mathbf{E}f_{n+1}^*(1, \sigma_1^1(n+1))=\pi_1+\theta^2x_n.
	$$
	Therefore, it follows from equation \eqref{MLE} that
	\begin{eqnarray*}
		\pi_1+\theta^2x_n\leq\Delta_{n+1}
		\leq\pi_1+\pi_1^{1/2}x_{n+1}^{1/2},
	\end{eqnarray*}
	namely,
	\begin{equation}
	\label{nftf} x_{n+1}\geq \frac{1}{\pi_1}\theta^4x_n^2\geq
	\varrho^4x_n^2.
	\end{equation}
	
	Next choosing $\varepsilon=\varrho^2$, equation \eqref{lemLA} indicates that there exists a $\delta=\delta(\pi,
	\varepsilon)>0$, such that if $x_n<\delta$ then
	$$
	x_{n+1}\geq (d\theta^2-\varepsilon)x_n\geq
	(d-1)\varrho^2x_n\geq\varrho^2x_n.
	$$
	On the other hand, if $x_n\geq \delta$, equation \eqref{nftf} becomes
	$x_{n+1}\geq\varrho^4\delta x_n$. Finally taking
	$\gamma=\min\{\varrho^2, \varrho^4\delta\}$ completes the proof.
\end{proof}

Actually, it can be seen from equation \eqref{explicit} that the estimates of $R$
and $S$ would play a key role in the recursive expression of
$x_{n+1}$, hence we will verify that
$\frac{\pi_1Z_1}{\pi_1+\pi_2Z_2}$ and $\frac{z_n}{x_n}$ are both
sufficiently around $\pi_1$, analogous to the concentration analysis result
in \cite{sly2009reconstruction}. 
In the following lemma, we firstly establish a technical uniqueness result where the set of vertices which can be conditioned is limited to a set of $k$ vertices.

\begin{lemma}
	\label{effectlittle} For any $\varepsilon>0$ and positive integer
	$k$, there exists $M=M(\pi, \varepsilon, k)$, such that for any
	collection of vertices $v_1,\ldots, v_k\in L(M)$,
	$$
	\sup_{i_1,\ldots,
		i_k\in\mathcal{C}}\left|\mathbf{P}(\sigma_\rho=1\mid\sigma_{v_j}=i_j,
	1\leq j\leq k)-\pi_1\right|\leq \varepsilon
	$$
\end{lemma}
\begin{proof}
	Denote the entries of the transition matrix at distance $s$ as
	$$
	U_s=M_{1, 1}^s \quad\textup{and}\quad V_s=M_{2, 2}^s,
	$$
	and it is natural that $M_{1, 2}^s=1-U_s$ and $M_{2, 1}^s=1-V_s$. As
	a result, it follows that
	$$
	\left\{\begin{array}{ll} U_s=M_{11}U_{s-1}+M_{12}(1-V_{s-1})
	\\
	V_s=M_{21}(1-U_{s-1})+M_{22}V_{s-1},
	\end{array}
	\right.
	$$
	which yields a second order recursive formula
	$$
	U_s-(1+\theta)U_{s-1}+\theta U_{s-2}=0
	$$
	with the initial conditions $U_0=1$ and
	$U_1=M_{11}=\pi_1+\pi_2\theta$. Then the general solutions are given by
	\begin{eqnarray*} U_s=\pi_1+\pi_2\theta^s\quad\textup{and}\quad
		V_s=\pi_2+\pi_1\theta^s.
	\end{eqnarray*}
	Consequently, under the condition of $d\theta^2\leq 1$, we have
	$$
	\pi_1-\pi_2d^{-s/2}\leq M_{1, 1}^s\leq \pi_1+\pi_2d^{-s/2},
	$$
	$$
	\pi_2-\pi_1d^{-s/2}\leq M_{2, 2}^s\leq \pi_2+\pi_1d^{-s/2},
	$$
	$$
	\pi_2-\pi_2d^{-s/2}\leq M_{1, 2}^s\leq \pi_2+\pi_2d^{-s/2},
	$$
	$$
	\pi_1-\pi_1d^{-s/2}\leq M_{2, 1}^s\leq \pi_1+\pi_1d^{-s/2}.
	$$
	For fixed $\pi$, $d$ and $k$, define
	$$
	B(s):=\max\left\{\frac{\pi_1+\pi_2d^{-s/2}}{\pi_1-\pi_2d^{-s/2}},
	\frac{\pi_2+\pi_1d^{-s/2}}{\pi_2-\pi_1d^{-s/2}},
	\frac{1+d^{-s/2}}{1-d^{-s/2}}\right\}
	$$
	and let $N=N(\pi, \varepsilon, k)$ be a sufficiently large integer
	such that
	$$
	B^k(N)\leq 1+\varepsilon,
	$$
	where the last inequality holds by the fact that $d^{-s/2}\leq2^{-s/2}\to0$ as $s\to\infty$ which implies
	$B(s)\to1$ uniformly for all $d$.

	Now fix an integer $M$ such that $M>kN$ and choose any
	$v_1,\ldots,v_k\in L(M)$. For $0\leq \ell\leq M$, define $n(\ell)$ as the number of vertices of distance $\ell$ from the root with a
	decedent in the set $\{v_1,\ldots,v_k\}$, that is
	$$
	n(\ell)=\#\left\{v\in L(\ell): \left|\mathbb{T}_v\cap
	\{v_1,\ldots,v_k\}\right|>0\right\}.
	$$
	Then according to the definition, it is trivial to see that
	$n(\ell)$ is an increasing integer valued function with respect to
	$\ell$ from $n_0=1$ to $n_M=k$, which, by the pigeonhole principle,
	implies that there must exist some $\ell$ such that
	$n(\ell)=n(\ell+N)$. Next, denote $\{w_1,\ldots,w_{n(\ell)}\}$ and
	$\{\overline{w}_1,\ldots,\overline{w}_{n(\ell)}\}$ as vertices in
	sets $\{v\in L(\ell+N): |\mathbb{T}_v\cap \{v_1,\ldots,v_k\}|>0\}$
	and $\{v\in L(\ell): |\mathbb{T}_v\cap \{v_1,\ldots,v_k\}|>0\}$
	respectively, such that $w_j$ is the descendent of $\overline{w}_j$,
	and then
	$$
	\mathbf{P}(\sigma_{w_j}=i_2\mid\sigma_{\overline{w}_j}=i_1)=M_{i_1,
		i_2}^N.
	$$
	By Bayes' Rule and the Markov random field property, for any
	$i_1,\ldots,i_{n(\ell)}\in \mathcal{C}$, we have
	\begin{equation*}
	\begin{aligned}
	&\frac{\mathbf{P}(\sigma_\rho=1\mid\sigma_{w_j}=i_j, 1\leq j\leq
		n(\ell))}{\mathbf{P}(\sigma_\rho=2\mid\sigma_{w_j}=i_j, 1\leq j\leq
		n_(\ell))}
	\\
	=&\frac{\pi_1}{\pi_2}\frac{\mathbf{P}(\sigma_{w_j}=i_j, 1\leq j\leq
		n(\ell)\mid\sigma_\rho=1)}{\mathbf{P}(\sigma_{w_j}=i_j, 1\leq j\leq
		n(\ell)\mid\sigma_\rho=2)}
	\\
	=&\frac{\pi_1}{\pi_2}\frac{\sum_{h_1,\ldots,h_{n(\ell)}\in
			\mathcal{C}}\mathbf{P}(\forall j\ \sigma_{w_j}=i_j\mid\forall j\
		\sigma_{\overline{w}_j}=h_j)\mathbf{P}(\forall j\
		\sigma_{\overline{w}_j}=h_j\mid\sigma_\rho=1)}{\sum_{h_1,\ldots,h_{n(\ell)}\in
			\mathcal{C}}\mathbf{P}(\forall j\ \sigma_{w_j}=i_j\mid\forall j\
		\sigma_{\overline{w}_j}=h_j)\mathbf{P}(\forall j\
		\sigma_{\overline{w}_j}=h_j\mid\sigma_\rho=2)}
	\\
	=&\frac{\pi_1}{\pi_2}\frac{\sum_{h_1,\ldots,h_{n(\ell)}\in
			\mathcal{C}}\mathbf{P}(\forall j\
		\sigma_{\overline{w}_j}=h_j\mid\sigma_\rho=1)\prod_{j=1}^{n(\ell)}M_{h_ji_j}^N}{\sum_{h_1,\ldots,h_{n(\ell)}\in
			\mathcal{C}}\mathbf{P}(\forall j\
		\sigma_{\overline{w}_j}=h_j\mid\sigma_\rho=2)\prod_{j=1}^{n(\ell)}M_{h_ji_j}^N}
	\\
	\leq&\frac{\pi_1}{\pi_2}B^{n(\ell)}(N)\frac{\sum_{h_1,\ldots,h_{n(\ell)}\in
			\mathcal{C}}\mathbf{P}(\forall j\
		\sigma_{\overline{w}_j}=h_j\mid\sigma_\rho=1)}{\sum_{h_1,\ldots,h_{n(\ell)}\in
			\mathcal{C}}\mathbf{P}(\forall j\
		\sigma_{\overline{w}_j}=h_j\mid\sigma_\rho=2)}
	\\
	\leq&\frac{\pi_1}{\pi_2}B^{k}(N)
	\\
	\leq&\frac{\pi_1}{\pi_2}(1+\varepsilon),
	\end{aligned}
	\end{equation*}
	which implies that
	$$
	\pi_1-\varepsilon\leq\mathbf{P}(\sigma_\rho=1\mid\sigma_{w_j}=i_j,
	1\leq j\leq n(\ell))\leq\pi_1+\varepsilon.
	$$
	
 Hence, for the reason that $\sigma_\rho$ is conditionally independent of the
	collection $\sigma_{v_1},\ldots,\sigma_{v_k}$ given
	$\sigma_{w_1},\ldots,\sigma_{w_{n(\ell)}}$, one has
	\begin{equation*}
	\begin{aligned}
	\sup_{i_1,\ldots,\
		i_k\in\mathcal{C}}\left|\mathbf{P}(\sigma_\rho=1\mid\sigma_{v_j}=i_j,
	1\leq j\leq k)-\pi_1\right|
	&\leq\sup_{i_1,\ldots,\
		i_{n(\ell)}\in\mathcal{C}}\left|\mathbf{P}(\sigma_\rho=1\mid\sigma_{w_j}=i_j,
	1\leq j\leq n(\ell))-\pi_1\right|
	\\
	&\leq\varepsilon.
	\end{aligned}
	\end{equation*}
\end{proof}

\begin{lemma}
	\label{concentration} Assume $|\theta|>\varrho$ for some
	$\varrho>0$. Given arbitrary $\varepsilon, \alpha> 0$, there exist
	constants $C=C(\pi, \varepsilon, \alpha, \varrho)>0$ and $N=N(\pi,
	\varepsilon, \alpha)$, such that whenever $n\geq N$,
	$$
	\mathbf{P}\left(\left|\frac{\pi_1Z_1}{\pi_1Z_1+\pi_2Z_2}-\pi_1\right|>\varepsilon\right)\leq
	Cx_n^\alpha.
	$$
\end{lemma}
\begin{proof}
	Fix $k$ an integer with $k>\alpha$. 
	Choose $M$ to hold with bound $\varepsilon/2$ in Lemma \ref{effectlittle}. Let
	$v_1,\ldots,v_{|L(M)|}$ denote the vertices in $L(M)$ and define
	$$
	W(v)=f_{n+1-M}(1, \sigma_v^1(n+1)),
	$$
	where $\sigma_v^1(n+1)$ denotes the spins of vertices in
	$\mathbb{T}_v\cap L(n+1)$. Then $W(v)$ would be distributed as
	\begin{eqnarray}
	\label{distribution} W(v)\sim \left\{\begin{array}{ll} X^+(n+1-M)&
	\quad \textup{if}\ \sigma_v^1=1,
	\\
	\
	\\
	1-X^-(n+1-M)& \quad\textup{if}\ \sigma_v^1=2.
	\end{array}
	\right.
	\end{eqnarray}
	The recursion formula in equation \eqref{recursion} together with the fact that
	$1-W(v)=f_{n+1-M}(2, \sigma_v^1(n+1))$, yield a function
	$$
	H(W_1,\ldots,W_{|L(M)|})=f_{n+1}(1,
	\sigma^1(n+1))=\frac{\pi_1Z_1}{\pi_1Z_1+\pi_2Z_2},
	$$
	where $W_i=W(v_i)$ for $1\leq i\leq |L(M)|$. There is no difficulty
	in finding that when all the entries $W_i$ are identically $\pi_1$ one has
	$$H(W_1,\ldots,W_{|L(M)|})=\pi_1,$$ and $H$ is a continuous
	function of the vector $(W_i)_{1\leq i\leq |L(M)|}$. Therefore, by
	Lemma \ref{effectlittle}, if there are at most $k$ vertices in $L(M)$
	such that $W(v)\neq \pi_1$,  then
	$$
	|H(W_1,\ldots,W_{|L(M)|})-\pi_1|<\frac\varepsilon2,
	$$
	and there exists some $\delta=\delta(\pi, \varepsilon)>0$ such
	that if
	$$
	\#\left\{v\in L(M): |W(v)-\pi_1|>\delta\right\}\leq k
	$$
	then
	$$
	|H(W_1,\ldots,W_{|L(M)|})-\pi_1|<\varepsilon.
	$$
	Next, by the Chebyshev's inequality together
	with equation \eqref{distribution}, the following result holds:
	\begin{eqnarray*}
		\mathbf{P}(|W(v)-\pi_1|>\delta)&\leq&
		\delta^{-2}[\mathbf{E}(X^+(n+1-M)-\pi_1)^2+\mathbf{E}(X^-(n+1-M)-\pi_2)^2]
		\\
		&\leq&\frac{\delta^{-2}}{\min\left\{1,
			\frac{\pi_2}{\pi_1}\right\}}x_{n+1-M}.
	\end{eqnarray*}
	Random variables $|W(v)-\pi_1|$ for distinct $v$ are
	conditionally independent given $\sigma(M)$, so there exist suitable
	constants $C(\pi, \varepsilon, \alpha, \varrho)$ and $N(\pi,
	\varepsilon, \alpha)$, such that when $n\geq N$, one has
	\begin{equation*}
	\begin{aligned}
	\mathbf{P}\left(\left|\frac{\pi_1Z_1}{\pi_1Z_1+\pi_2Z_2}-\pi_1\right|>\varepsilon\right)
	&\leq\mathbf{P}(\#\left\{v\in L(M): |W(v)-\pi_1|>\delta\right\}> k)
	\\
	&=\sum_A\mathbf{P}(\#\left\{v\in L(M): |W(v)-\pi_1|>\delta\right\}>
	k\mid\sigma(M)=A)\mathbf{P}(\sigma(M)=A)
	\\
	&\leq\sum_A\mathbf{P}\left[\mathbf{B}\left(|L(M)|,
	\frac{\delta^{-2}}{\min\left\{1,
		\frac{\pi_2}{\pi_1}\right\}}x_{n+1-M}\right)>k\right]\mathbf{P}(\sigma(M)=A)
	\\
	&\leq C'(\pi, \varepsilon, \alpha, \varrho)x_{n+1-M}^\alpha
	\\
	&\leq Cx_n^\alpha,
	\end{aligned}
	\end{equation*}
	where $\mathbf{B}\left(\cdot, \cdot\right)$ denotes the binomial
	distribution and the last inequality holds due to Lemma \ref{ndtf}.
\end{proof}

Now, we are able to bound $S$ and $R$ in equation \eqref{sly2009reconstruction} using the preceding concentration results.
\begin{proposition}
	\label{estimateforS} Assume $|\theta|>\varrho$ for some $\varrho>0$.
	For any $\varepsilon>0$, there exist $N=N(\pi, \varepsilon)$ and
	$\delta=\delta(\pi, \varepsilon, \varrho)>0$, such that if $n\geq N$
	and $x_n\leq\delta$ then $|S|\leq\varepsilon x_n^2$.
\end{proposition}
\begin{proof}
	For any $\eta>0$, using the Cauchy-Schwartz inequality and by 
	Lemma \ref{concentration}, one has
	\begin{eqnarray*}
		|S|&=&\left|\mathbf{E}(\pi_1Z_1+\pi_2Z_2-1)^2\left(\frac{\pi_1Z_1}{\pi_1Z_1+\pi_2Z_2}-\pi_1\right)\right|
		\\
		&\leq&\mathbf{E}\left((\pi_1Z_1+\pi_2Z_2-1)^2\left|\frac{\pi_1Z_1}{\pi_1Z_1+\pi_2Z_2}-\pi_1\right|;\left|\frac{\pi_1Z_1}{\pi_1Z_1+\pi_2Z_2}-\pi_1\right|\leq
		\eta\right)
		\\
		&&+\mathbf{E}\left((\pi_1Z_1+\pi_2Z_2-1)^2\left|\frac{\pi_1Z_1}{\pi_1Z_1+\pi_2Z_2}-\pi_1\right|;\left|\frac{\pi_1Z_1}{\pi_1Z_1+\pi_2Z_2}-\pi_1\right|>
		\eta\right)
		\\
		&\leq&\eta
		\mathbf{E}(\pi_1Z_1+\pi_2Z_2-1)^2+\mathbf{E}(\pi_1Z_1+\pi_2Z_2-1)^2\mathbf{I}\left(\left|\frac{\pi_1Z_1}{\pi_1Z_1+\pi_2Z_2}-\pi_1\right|>
		\eta\right)
		\\
		&\leq&\eta
		\mathbf{E}(\pi_1Z_1+\pi_2Z_2-1)^2+\mathbf{P}\left(\left|\frac{\pi_1Z_1}{\pi_1Z_1+\pi_2Z_2}-\pi_1\right|>
		\eta\right)^{1/2}\left[\mathbf{E}(\pi_1Z_1+\pi_2Z_2-1)^4\right]^{1/2}.
	\end{eqnarray*}
	Also, it follows from equation \eqref{square} and Lemma \ref{Taylor}
	respectively that
	$$
	\mathbf{E}(\pi_1Z_1+\pi_2Z_2-1)^2\leq C_1(\pi)x_n^2
	$$
	and
	$$
	(\mathbf{E}(\pi_1Z_1+\pi_2Z_2-1)^4)^{1/2}\leq C_2(\pi).
	$$
	Taking $\alpha=6$ in Lemma \ref{concentration}, there exist
	$C_3=C_3(\pi, \eta, \varrho)$ and $N=N(\pi, \eta)$, such that if
	$n\geq N$ then
	$$
	\mathbf{P}\left(\left|\frac{\pi_1Z_1}{\pi_1Z_1+\pi_2Z_2}-\pi_1\right|>\eta\right)\leq
	C_3^2x_n^6.
	$$
	Finally taking $\eta=\varepsilon/(2C_1)$ and
	$\delta=\varepsilon/(2C_2C_3)$, we have that if $n\geq N$ and
	$x_n\leq\delta$ then
	\begin{eqnarray*}
		|S|\leq\eta C_1x_n^2+C_2C_3x_n^3\leq\varepsilon x_n^2.
	\end{eqnarray*}
\end{proof}

\begin{proposition}
	\label{concentrationforz} Assume $|\theta|>\varrho$ for some
	$\varrho>0$. For any $\varepsilon>0$, there exist $N=N(\pi,
	\varepsilon)$ and $\delta=\delta(\pi, \varepsilon, \varrho)$, such
	that if $n\geq N$ and $x_n\leq\delta$ then
	$$
	\left|\frac{z_n}{x_n}-\pi_1\right|\leq \varepsilon.
	$$
\end{proposition}
\begin{proof}
	Plugging $a=(Z_1-Z_2)^2$, $r=(\pi_1Z_1+\pi_2Z_2)^2-1$ and $s=1$ in the
	identity equation \eqref{identity}, we have
	\begin{eqnarray*}
		z_{n+1}&=&\mathbf{E}\left(\frac{\pi_1Z_1}{\pi_1Z_1+\pi_2Z_2}-\pi_1\right)^2
		\\
		&=&\pi_1^2\pi_2^2\mathbf{E}\frac{(Z_1-Z_2)^2}{1+(\pi_1Z_1+\pi_2Z_2)^2-1}
		\\
		&=&\pi_1^2\pi_2^2\left\{\mathbf{E}(Z_1-Z_2)^2-\mathbf{E}(Z_1-Z_2)^2[(\pi_1Z_1+\pi_2Z_2)^2-1]+\mathbf{E}[(\pi_1Z_1+\pi_2Z_2)^2-1]^2\frac{(Z_1-Z_2)^2}{(\pi_1Z_1+\pi_2Z_2)^2}\right\}.
	\end{eqnarray*}
	Next we will estimate these expectation terms one by one with the $O_\pi$-constants depend only on $\pi$:
	\begin{equation*}
	\begin{split}
		\mathbf{E}(Z_1-Z_2)^2
		=&\mathbf{E}(Z_1^2+Z_2^2-2Z_1Z_2)
		\\
		=&1+d\left[\frac{\theta^2}{\pi_1}(3-\theta)x_n+\frac{\theta^3}{\pi_1^2}z_n\right]
		+1+d\left\{\frac{\theta^2}{\pi_2}\left[\frac{\pi_1}{\pi_2}(1-\theta)-2\right]x_n+\frac{\theta^3}{\pi_2^2}z_n\right\}
		\\
		&-2\left\{1+d\left[\theta^2\left(\frac{1}{\pi_1}+\frac{\theta-2}{\pi_2}\right)x_n-\frac{\theta^3}{\pi_1\pi_2}z_n\right]\right\}+O_\pi(x_n^2)
		\\
		=&d\theta^2\left(\frac{1-\theta}{\pi_1\pi_2^2}x_n+\frac{\theta}{\pi_1^2\pi_2^2}
		z_n\right)+O_\pi(x_n^2),
			\end{split}
		\end{equation*}
			\begin{equation*}
		\begin{split}
		&\mathbf{E}(Z_1-Z_2)^2[(\pi_1Z_1+\pi_2Z_2)^2-1]=O_\pi(x_n^2),\\
        &\pi_1^2\pi_2^2\mathbf{E}[(\pi_1Z_1+\pi_2Z_2)^2-1]^2\frac{(Z_1-Z_2)^2}{(\pi_1Z_1+\pi_2Z_2)^2}
		\leq \mathbf{E}[(\pi_1Z_1+\pi_2Z_2)^2-1]^2=O_\pi(x_n^2),
	\end{split}
	\end{equation*}
where we used the fact that $
\pi_1^2\pi_2^2(Z_1-Z_2)^2/(\pi_1Z_1+\pi_2Z_2)^2\leq 1$ in the last inequality.  
	Therefore, the recursion formula of $z_{n+1}$ can be written as
	$$
	z_{n+1}=d\theta^2[\pi_1(1-\theta)x_n+\theta z_n]+O_\pi(x_n^2).
	$$
In the rest of the proof, we let $\{C_i\}_{i=1,2,3,4}$ be constants depend only on $\pi$. It follows from equation \eqref{lemLA} that
	$$
	\left|\frac{d\theta^2x_n}{x_{n+1}}-1\right|\leq C_1
	\frac{x_n^2}{x_{n+1}},
	$$ and in view of $d\theta^2\geq1/2$,
	there exists $\delta_1=\delta_1(\pi)>0$, such that if
	$x_n\leq\delta_1$ then
	\begin{eqnarray}
	\label{5.3}
	\frac{x_n}{x_{n+1}}&=&\frac{x_n}{d\theta^2x_n+O_\pi(x_n^2)}
	\leq\frac{x_n}{9d\theta^2x_n/10}
	=\frac{10}{9}\frac{1}{d\theta^2}\leq\frac{20}{9}.
	\end{eqnarray}
	Consequently,
	\begin{equation}
	\label{iteration}
	\begin{aligned}
	\left|\frac{z_{n+1}}{x_{n+1}}-\left[\pi_1(1-\theta)+\theta\frac{z_n}{x_n}\right]\right|
	&=\left|\frac{z_{n+1}}{x_{n+1}}-\frac{d\theta^2x_n}{x_{n+1}}\left[\pi_1(1-\theta)+\theta
	\frac{z_n}{x_n}\right]\right|
	+\left|\left(\frac{d\theta^2x_n}{x_{n+1}}-1\right)\left[\pi_1(1-\theta)+\theta
	\frac{z_n}{x_n}\right]\right|
	\\
	&\leq C_2\frac{x_n^2}{x_{n+1}}+C_3\frac{x_n^2}{x_{n+1}}
	\\
	&\leq C_4x_n.
	\end{aligned}
	\end{equation}
For any
	$k\in\mathbb{N}$, by equation \eqref{lemLA}, there exists a
	$\delta_2=\delta_2(\pi, k)$, such that if $x_n\leq\delta_2$ then
	$x_{n+\ell}\leq2\delta_2\leq\delta_1$ for any $1\leq\ell\leq k$. Now
	iterating $k$ times the inequality \eqref{iteration} yields
	\begin{equation*}
	\begin{aligned}
	&\left|\frac{z_{n+k}}{x_{n+k}}-\left[\pi_1(1-\theta^k)+\theta^k\frac{z_n}{x_n}\right]\right|
	\\
	&\leq\sum_{\ell=1}^k\left|\pi_1(1-\theta^{k-\ell})+\theta^{k-\ell}\frac{z_{n+\ell}}{x_{n+\ell}}-\pi_1(1-\theta^{k-\ell+1})-\theta^{k-\ell+1}\frac{z_{n+\ell-1}}{x_{n+\ell-1}}\right|
	\\
	&\leq\sum_{\ell=1}^k|\theta|^{k-\ell}\left|\frac{z_{n+\ell}}{x_{n+\ell}}-\left[\pi_1(1-\theta)+\theta\frac{z_{n+\ell-1}}{x_{n+\ell-1}}\right]\right|
	\\
	&\leq C_4\sum_{\ell=1}^k|\theta|^{k-\ell}x_{n+\ell-1}.
	\end{aligned}
	\end{equation*}
	Therefore, noting that $|\theta|\leq d^{-1/2}\leq 2^{-1/2}$, and taking $k=k(\varepsilon)$ large enough and
	$\delta_3=\delta_3(\pi, \varepsilon, k)=\delta_3(\pi, \varepsilon)$
	sufficiently small, we obtain that if $x_n\leq\delta_3$ then
	\begin{equation}
	\label{concentrationk}
	\left|\frac{z_{n+k}}{x_{n+k}}-\pi_1\right|\leq|\theta|^k+2\delta_3
	C_4\sum_{\ell=1}^k|\theta|^{k-\ell}=|\theta|^k+2\delta_3
	C_4\frac{1-|\theta|^k}{1-|\theta|}\leq\varepsilon,
	\end{equation}
	where the first inequality relies on the fact that $|z_n/x_n-\pi_1|<1$. At last, choosing $N=N(\pi, \varepsilon)>k$ and $\delta=\delta(\pi,
	\varepsilon, \varrho)=\gamma^k\delta_3$, and noting that by
	Lemma \ref{ndtf} if $n\geq N$ and $x_n\leq\delta$ then
	$x_{n-k}\leq\gamma^{-k}x_n\leq\delta_3$, the previous
	result in equation \eqref{concentrationk} completes the proof.
\end{proof}

\section {Proof of Theorem \ref{reconstruction}}
\label{Proof_of_Theorem_1.1}
To accomplish the proof, it suffices to show that when $d\theta^2$
is close enough to $1$, $x_n$ does not converge to $0$. For convenience, we suppose that $d\theta^2\geq1/2$. For any
fixed $d$ and $\pi$, there is
$|\theta|\geq(2d)^{-1/2}$, and we take $\varrho=(2d)^{-1/2}$ in
Lemma \ref{ndtf} to generate $\gamma=\gamma(\pi, d)>0$. When $1-6\pi_1\pi_2>0$, by
Proposition~\ref{estimateforS} and Proposition \ref{concentrationforz},
there exist $N=N(\pi)$ and $\delta=\delta(\pi, d)>0$, such that if
$n\geq N$ and $x_n\leq\delta$, then the remainders
in equation \eqref{explicit} could be evaluated respectively as
\begin{equation}
\label{6.1}
|R|\leq\frac{1}{6}\frac{1-6\pi_1\pi_2}{\pi_1\pi_2^2}\frac{d(d-1)}{2}\theta^4x_n^2
\end{equation}
and
\begin{equation}
\label{6.2}
|S|\leq\frac{1}{6}\frac{1-6\pi_1\pi_2}{\pi_1\pi_2^2}\frac{d(d-1)}{2}\theta^4x_n^2.
\end{equation}
As a consequence, 
\begin{equation}
\label{6.4} x_{n+1}\geq
d\theta^2x_n+\frac{1}{2}\frac{(1-6\pi_1\pi_2)}{\pi_1\pi_2^2}\frac{d(d-1)}{2}\theta^4x_n^2.
\end{equation}
Furthermore, in light of $x_0=1-\pi_1=\pi_2$ and Lemma \ref{ndtf},
for all $n$ we have
\begin{equation}
\label{6.5} x_n\geq \pi_2\gamma^n.
\end{equation}
Define $\varepsilon=\varepsilon(\pi, d)=\min\{\pi_2\gamma^{N},
\delta\gamma\}>0$. Then equation \eqref{6.5} implies that $x_n\geq
\varepsilon$ when $n\leq N$. Next, by choosing suitable
$|\theta|<d^{-1/2}$, we achieve
\begin{eqnarray}
\label{6.6}
d\theta^2+\frac{1}{2}\frac{(1-6\pi_1\pi_2)}{\pi_1\pi_2^2}\frac{d(d-1)}{2}\theta^4\varepsilon\geq1,
\end{eqnarray}
for the reason that $\varepsilon$ is independent of $\theta$. Now, suppose
$x_n\geq\varepsilon$ for some $n\geq N$. If $x_n\geq
\gamma^{-1}\varepsilon$, then Lemma \ref{ndtf} gives $x_{n+1}\geq
\gamma x_n\geq \varepsilon$. If $\varepsilon\leq x_n\leq
\gamma^{-1}\varepsilon\leq\delta$, then by equation \eqref{6.4}
and equation \eqref{6.6}, we have
\begin{equation*}
	x_{n+1}\geq
	d\theta^2x_n+\frac{1}{2}\frac{(1-6\pi_1\pi_2)}{\pi_1\pi_2^2}\frac{d(d-1)}{2}\theta^4x_n^2
	\geq x_n\left[d\theta^2+\frac{1}{2}\frac{(1-6\pi_1\pi_2)}{\pi_1\pi_2^2}\frac{d(d-1)}{2}\theta^4\varepsilon\right]
	\geq x_n
	\geq\varepsilon.
\end{equation*}
Hence it can be shown by induction that $x_n\geq \varepsilon$ for all $n$, namely, the Kesten-Stigum bound is not tight.

\section{Large Degree Asymptotics}
\subsection{Gaussian Approximation}
\indent For $1\leq j \leq d$, define
$$
U_j=\log\left[1+\frac{\theta}{\pi_1}(Y_j-\pi_1)\right]\quad\textup{and}\quad
V_j=\log\left[1-\frac{\theta}{\pi_2}(Y_j-\pi_1)\right].
$$
\begin{lemma}
	\label{meansandvariances} There exist positive constants $C=C(\pi)$
	and $D=D(\pi)$, such that when $d>D$,
	$$
	\left|d\mathbf{E}U_j-\frac{d\theta^2}{2\pi_1}x_n\right|\leq
	Cd^{-1/2},\quad
	\left|d\mathbf{E}V_j+\frac{1+\pi_2}{2\pi_2^2}d\theta^2x_n\right|\leq
	Cd^{-1/2},
	$$
	$$
	\left|d\mathbf{Var}(U_j)-\frac{d\theta^2}{\pi_1}x_n\right|\leq
	Cd^{-1/2}, \quad
	\left|d\mathbf{Var}(V_j)-\frac{\pi_1}{\pi_2^2}d\theta^2x_n\right|\leq
	Cd^{-1/2},
	$$
	$$
	\left|d\mathbf{Cov}(U_j,
	V_j)+\frac{d\theta^2}{\pi_2}x_n\right|\leq Cd^{-1/2}.
	$$
\end{lemma}
\begin{proof}
	Starting with the Taylor series expansion of $\log(1+w)$, there
	exists a constant $W > 0$, such that when $|w| < W$,
	\begin{equation}
	\label{log} \left|\log(1 + w)- w + \frac{w^2}{2}\right|\leq|w|^3.
	\end{equation}
	Taking $D=D(\pi)$ sufficiently large, when $d>D$, we have that $|\theta|\leq
	d^{-\frac12}$ is small enough to guarantee equation \eqref{log} for
	$w=\theta(Y_j-\pi_1)/\pi_1$ and then
	\begin{eqnarray*}
		\left|\mathbf{E}U_j-\frac{\theta^2}{2\pi_1}x_n\right|
		&\leq&\mathbf{E}\left|U_j-w+\frac{w^2}{2}\right|
	+\left|\mathbf{E}w-\mathbf{E}\frac{w^2}{2}-\frac{\theta^2}{2\pi_1}x_n\right|
		\\
		&\leq&\mathbf{E}\frac{\theta^3}{\pi_1^3}|Y_j-\pi_1|^3+\frac{\theta^3}{2\pi_1^2}|z_n-\pi_1x_n|
		\\
		&\leq&\frac{\theta^3}{\pi_1^3}+\frac{\theta^3}{2\pi_1^2}
		\\
		&\leq&C(\pi)d^{-3/2}
	\end{eqnarray*}
	for some constant $C=C(\pi)$, where the third inequality follows
	from $0\leq z_n\leq x_n\leq 1$. The rest estimates follow similarly.
\end{proof}
Next, define a 2-dimensional vector
$\mu=(\mu_1, \mu_2)$ with
$\mu_1=\frac{1}{2\pi_1},
\mu_2=-\frac{1+\pi_2}{2\pi_2^2}
$, 
and a $2\times 2$-covariance matrix
$$
\Sigma=\left(
\begin{array}{cc}
\frac{1}{\pi_1} & -\frac{1}{\pi_2} \\
-\frac{1}{\pi_2} & \frac{\pi_1}{\pi_2^2} \\
\end{array}
\right),
$$
which is a positive semi-definite symmetric $2\times 2$-matrix. Let
$(G_1, G_2)$ possess the Gaussian distribution $\mathbf{N}(\mu,
\Sigma)$, then the following lemma can be established by the Central Limit Theorem, the Gaussian approximation and the Portmanteau
Theorem.
\begin{lemma}
	\label{Gaussian} Let $\psi : \mathbb{R}^2\mapsto \mathbb{R}$ be a
	differentiable bounded function. For any $\varepsilon>0$, there
	exists $D=D(\pi, \psi, \varepsilon)>0$, such that if $d>D$ then
	$$
	\left|\mathbf{E}\psi\left(\sum_{j=1}^dU_j,
	\sum_{j=1}^dV_j\right)-\mathbf{E}\psi(G_1, G_2)\right|\leq
	\varepsilon.
	$$
\end{lemma}
Next, define
$$
\psi(w_1, w_2)=\frac{\pi_1e^{w_1}}{\pi_1e^{w_1}+\pi_2e^{w_2}}
$$
and then
\begin{eqnarray*}
	x_{n+1}&=&\mathbf{E}\frac{\pi_1Z_1}{\pi_1Z_1+\pi_2Z_2}-\pi_1
	\\
	&=&\mathbf{E}\frac{\pi_1\exp\left(\sum_{j=1}^dU_j\right)}{\pi_1\exp\left(\sum_{j=1}^dU_j\right)+\pi_2\exp\left(\sum_{j=1}^dV_j\right)}-\pi_1
	\\
	&=&\mathbf{E}\psi\left(\sum_{j=1}^dU_j,
	\sum_{j=1}^dV_j\right)-\pi_1.
\end{eqnarray*}
If $(W_1, W_2)$ has the Gaussian distribution $\mathbf{N}(0, \Sigma)$,
then $(s\mu_1+\sqrt{s}W_1, s\mu_2+\sqrt{s}W_2)$ is distributed
according to $\mathbf{N}(s\mu, s\Sigma)$. At last, define
\begin{eqnarray*}
	g(s)=\mathbf{E}\psi(s\mu_1+\sqrt{s}W_1, s\mu_2+\sqrt{s}W_2)-\pi_1
=\mathbf{E}\frac{\pi_1\exp(s\mu_1+\sqrt{s}W_1)}{\pi_1\exp(s\mu_1+\sqrt{s}W_1)+\pi_2\exp(s\mu_2+\sqrt{s}W_2)}-\pi_1.
\end{eqnarray*}

Therefore, Lemma~\ref{Gaussian} implies the
following approximation to $x_{n+1}$.
\begin{lemma}
	\label{Gaussianapproximation} For arbitrary $\varepsilon>0$, there
	exists a $D=D(\pi, \varepsilon)>0$, such that whenever $d>D$,
	$$
	\big{|}x_{n+1}-g(d\theta^2x_n)\big{|}\leq \varepsilon.
	$$
\end{lemma}

\subsection{Asymptotic Estimation of the Reconstruction Threshold}
\indent In order to estimate $x_{n+1}$, it suffices to investigate the
properties of $g(s)$ on the interval $[0, \pi_2]$, considering that $0\leq
x_n\leq\pi_2$ and $d\theta^2\leq1$.
\begin{lemma}
	\label{increasing} The function $g(s)$ is continuously
	differentiable and increasing on the interval $(0, \pi_2]$.
\end{lemma}
\begin{proof}
	When $s>0$, it is concluded that
	\begin{equation*}
	\begin{aligned}
	&\mathbf{E}\left|\frac{\partial}{\partial
		s}\frac{\pi_1\exp(s\mu_1+\sqrt{s}W_1)}{\pi_1\exp(s\mu_1+\sqrt{s}W_1)+\pi_2\exp(s\mu_2+\sqrt{s}W_2)}\right|
	\\
	&=\mathbf{E}\left|\frac{\partial}{\partial
		s}\left(1+\frac{\pi_2}{\pi_1}\exp(s(\mu_2-\mu_1)+\sqrt{s}(W_2-W_1))\right)^{-1}\right|
	\\
	&=\mathbf{E}\left|\frac{\frac{\pi_2}{\pi_1}\exp(s(\mu_2-\mu_1)+\sqrt{s}(W_2-W_1))}{\left(1+\frac{\pi_2}{\pi_1}\exp(s(\mu_2-\mu_1)+\sqrt{s}(W_2-W_1))\right)^2}\left(\mu_2-\mu_1+\frac{W_2-W_1}{2\sqrt{s}}\right)\right|
	\\
	&\leq\frac14\mathbf{E}\left|\mu_2-\mu_1+\frac{W_2-W_1}{2\sqrt{s}}\right|
	\\
	&<\infty,
	\end{aligned}
	\end{equation*}
	by the fact that
	$\left|\frac{\pi_2}{\pi_1}e^t \bigg/\left(1+\frac{\pi_2}{\pi_1}e^t\right)^2\right|\leq1/4$
	holds for any $t\in\mathbb{R}$. Then we establish the
	differentiability with respect to $s$.
	
	Now, let $(W_1', W_2')$ be an independent copy of $(W_1, W_2)$. Thus
	if $0\leq s'\leq s$, it is feasible to construct equivalent
	distributions such as
	$$
	\sqrt{s}(W_1, W_2)\; \sim \; \sqrt{s'}(W_1, W_2)+\sqrt{s-s'}(W_1', W_2').
	$$
	In view of $(W_1, W_2) \sim \mathbf{N}(0, \Sigma)$, it follows that
	$\mathbf{E}(W_2-W_1)=0$ and
	\begin{equation*}
	\mathbf{Var}(W_2-W_1)
		=\mathbf{E}W_2^2+\mathbf{E}W_1^2-2\mathbf{E}W_1W_2\\
		=\frac{1}{\pi_1}+\frac{\pi_1}{\pi_2^2}-2\left(-\frac{1}{\pi_2}\right)\\
		=\frac{1}{\pi_1\pi_2^2},
	\end{equation*}
	which implies that $W_2-W_1$ and $W_2'-W_1'$ are both distributed as
	$\mathbf{N}(0, a)$, with $a=1/\pi_1\pi_2^2$.
	
	Next, it is well known that if $W$ has the distribution
	$\mathbf{N}(\mu, \sigma^2)$, the expectation of the exponential
	random variable could be estimated as
	\begin{equation}
	\label{exponential} \mathbf{E}e^W=e^{\mu+\frac{\sigma^2}{2}},
	\end{equation}
	based on which, we are able to estimate the conditional
	expectation given $W_1$ and $W_2$:
	\begin{eqnarray*}
		\mathbf{E}\left[\exp(\sqrt{s'}(W_2-W_1)+\sqrt{s-s'}(W_2'-W_1'))\bigg |\{W_1,
		W_2\}\right] =
		\exp\left[\sqrt{s'}(W_2-W_1)+\frac{a}{2}(s-s')\right].
	\end{eqnarray*}
	Then applying Jensen's inequality, and considering that the function
	$(1+x)^{-1}$ is convex and
	$$\mu_2-\mu_1=-(1+\pi_2)/(2\pi_2^2)-1/(2\pi_1)=-1/(2\pi_1\pi_2^2)=-a/2,$$
	we have
	\begin{eqnarray*}
		g(s)
		&=&\mathbf{E}\left(1+\frac{\pi_2}{\pi_1}\exp(s(\mu_2-\mu_1)+\sqrt{s}(W_2-W_1))\right)^{-1}-\pi_1
		\\
		&=&\mathbf{E}\left(1+\frac{\pi_2}{\pi_1}\exp\left(\frac{-as}{2}+\sqrt{s'}(W_2-
		W_1)+\sqrt{s-s'}(W_2'-W_1')\right)\right)^{-1}-\pi_1
		\\
		&\geq&\mathbf{E}\left(1+\frac{\pi_2}{\pi_1}\exp\left(-\frac{as}{2}\right)\mathbf{E}\left[\exp(\sqrt{s'}(W_2-W_1)+\sqrt{s-s'}(W_2'-W_1'))\mid\{W_1,
		W_2\}\right]\right)^{-1}-\pi_1
		\\
		&=&\mathbf{E}\left(1+\frac{\pi_2}{\pi_1}\exp\left(-\frac{as'}{2}\right)\mathbf{E}\exp(\sqrt{s'}(W_2-W_1))\right)^{-1}-\pi_1
		\\
		&=&\mathbf{E}\left(1+\frac{\pi_2}{\pi_1}\exp(s'(\mu_2-\mu_1)+\sqrt{s'}(W_2-W_1))\right)^{-1}-\pi_1
		\\
		&=&g(s'),
	\end{eqnarray*}
	as desired.
\end{proof}

It is necessary to discuss the Taylor expansion of $g(s)$ in the small
neighborhoods of $s=0$.
\begin{lemma}
	\label{power} For small $s>0$, we have
	$$
	g(s)=s
	+\frac{1-6\pi_1\pi_2}{2\pi_1\pi_2^2}s^2+\frac{1-24\pi_1\pi_2+90\pi_1^2\pi_2^2}{6\pi_1^2\pi_2^4}s^3+O_\pi(s^4).
	$$
\end{lemma}
\begin{proof}
	Define $W=s(\mu_2-\mu_1)+\sqrt{s}(W_2-W_1)$. By the results in Lemma \ref{increasing}, it is apparent that $W\sim
	\mathbf{N}\left(-as/2, as\right)$. Therefore by equation \eqref{exponential}
	the following moments can be calculated:
	\begin{eqnarray*}
		&&\mathbf{E}(e^W-1)=e^{\frac{-as}{2}+\frac{as}{2}}-1=e^0-1=0,
		\\
		&&\mathbf{E}(e^W-1)^2=e^{as}-1=as+\frac{a^2s^2}{2}+\frac{a^3s^3}{6}+O(s^4),
		\\
		&&\mathbf{E}(e^W-1)^3=e^{3as}-3e^{as}+2=3a^2s^2+4a^3s^3+O(s^4),
		\\
		&&\mathbf{E}(e^W-1)^4=e^{6as}-4e^{3as}+6e^{as}-3=3a^2s^2+19a^3s^3+O(s^4),
		\\
		&&\mathbf{E}(e^W-1)^5=e^{10as}-5e^{6as}+10e^{3as}-10e^{as}+4=30a^3s^3+O(s^4),
		\\
		&&\mathbf{E}(e^W-1)^6=e^{15as}-6e^{10as}+15e^{6as}-20e^{3as}+15e^{as}-5=15a^3s^3+O(s^4),
		\\
		&&\mathbf{E}(e^W-1)^7=e^{21as}-7e^{15as}+21e^{10as}-35e^{6as}+35e^{3as}-21e^{as}+6=O(s^4).
	\end{eqnarray*}
	Next starting from the identity
	$$
	\frac{1}{1+\pi_2(e^W-1)}=\sum_{n=0}^6(-1)^n\pi_2^n(e^W-1)^n-\pi_2^7(e^W-1)^7\frac{1}{1+\pi_2(e^W-1)},
	$$
	we obtain the power series of $g(s)$ as
	\begin{eqnarray*}
		\frac{g(s)+\pi_1}{\pi_1}&=&\frac{1}{\pi_1}\mathbf{E}\frac{\pi_1\exp(s\mu_1+\sqrt{s}W_1)}{\pi_1\exp(s\mu_1+\sqrt{s}W_1)+\pi_2\exp(s\mu_2+\sqrt{s}W_2)}
		\\
		&=&\mathbf{E}\frac{1}{1+\pi_2(e^W-1)}
		\\
		&=&\mathbf{E}\left[\sum_{n=0}^6(-1)^{n}\pi_2^n(e^W-1)^n-\pi_2^7(e^W-1)^7\frac{1}{1+\pi_2(e^W-1)}\right]
		\\
		&=&1+\frac{1}{\pi_1}\left(s
		+\frac{1-6\pi_1\pi_2}{2\pi_1\pi_2^2}s^2+\frac{1-24\pi_1\pi_2+90\pi_1^2\pi_2^2}{6\pi_1^2\pi_2^4}s^3+O(s^4)\right),
	\end{eqnarray*}
	that is,
	$$
	g(s)=s
	+\frac{1-6\pi_1\pi_2}{2\pi_1\pi_2^2}s^2+\frac{1-24\pi_1\pi_2+90\pi_1^2\pi_2^2}{6\pi_1^2\pi_2^4}s^3+O(s^4).
	$$
\end{proof}

\subsection{Proof of Theorem \ref{large_degree}}
\label{Proof_of_Theorem_1.2}
In this section, we precisely rephrase Theorem \ref{large_degree} and give its rigorous proof.
\begin{theorem}
When $\pi_1\pi_2<\frac16$, define
$$\omega^*=\inf \{\omega>0: \; \exists s\in (0,\pi_2), \; g(\omega s)=s\}.
	$$
	Then $0<\omega^*<1$, and for any $\delta>0$ there exists a $D=D(\pi,
	\delta)$, such that if $d>D$ then the model has reconstruction when
	$d\theta^2\geq \omega^*+\delta$, but does not have reconstruction
	when $d\theta^2\leq \omega^*-\delta$. In other words,
$$\lim_{d\to\infty}d\left(\theta^\pm\right)^2=\omega^*.$$
\end{theorem}

\begin{proof}
It follows from Lemma~\ref{power} that when $1-6\pi_1\pi_2>0$, there exists $0<\bar{s}<\pi_2$ such that $g(\bar{s})>\bar{s}$. Moreover, noting that $g(0\cdot \bar{s})=g(0)=0<\bar{s}$, the Intermediate Value Theorem implies the existence of $0<\bar{\omega}<1$ such that $g(\bar{\omega}\bar{s})=\bar{s}$. Consequently, $\omega^*$ does exist and $0\leq \omega^*\leq\bar{\omega}<1$. Furthermore, for any $\omega^*<\omega<1$, it follows from Lemma~\ref{increasing} that
the set 
$\left\{0<s<\pi_2: g(\omega s)\geq s\right\}$
is a non-empty compact set bounded away from $0$. Then it is further established by the continuity of $g(s)$ that the set
$$\left\{0< s< \pi_2: g(\omega^*s)=s\right\}=\bigcap_{\omega*<\omega<1}\left\{0<s<\pi_2: g(\omega s)\geq s\right\}$$
is non-empty and compact. Hence it implies immediately that $0<\omega^*<1$. 

Next, taking $s^*\in \{0<s<\pi_2:
	g(\omega^*s)=s\}$ and considering $d\theta^2=\omega^*+\delta$, one has
	$$
	g\left[(\omega^*+\delta)\left(s^*\frac{\omega^*}{\omega^*+\delta}\right)\right]=g(s^*\omega^*)=s^*>s^*\frac{\omega^*}{\omega^*+\delta}.
	$$
	Define $\varepsilon=\varepsilon(\pi,
	\delta)=s^*-s^*\omega^*/(\omega^*+\delta)>0$. By
	Lemma \ref{Gaussianapproximation} there exists a $D=D(\pi,
	\varepsilon)=D(\pi, \delta)$, such that if $d>D$ and $x_n\geq
	s^*\omega^*/(\omega^*+\delta)$ then
	\begin{eqnarray*}
	x_{n+1}\geq g(d\theta^2x_n)-\varepsilon
	\geq g\left[(\omega^*+\delta)\left(s^*\frac{\omega^*}{\omega^*+\delta}\right)\right]-\left(s^*-s^*\frac{\omega^*}{\omega^*+\delta}\right)
	=s^*\frac{\omega^*}{\omega^*+\delta},
\end{eqnarray*}
where the second inequality follows from Lemma \ref{increasing}.
Consequently, it is shown by induction, and noting the initial value
$x_0=\pi_2>s^*$, that $x_n\geq s^*\omega^*/(\omega^*+\delta)$ for
all $n$, which further establishes reconstruction. At last,
Proposition 12 in \cite{mossel2001reconstruction} implies that the reconstruction is
solvable for any $d\theta^2\geq \omega^*+\delta$.

On the other hand, when $d\theta^2=\omega^*-\delta$, we have
$g(d\theta^2s)\leq d\theta^2s/\omega^*$. Taking
$\eta=\left(1-\omega^*\right)/2>0$ in equation \eqref{lemLA}, there exists
a constant $\zeta=\zeta(\pi)$, such that if $x_n<\zeta$ then
$$ x_{n+1}\leq d\theta^2x_n+\eta x_n\leq
\frac{1}{2}\left(1+\omega^*\right)x_n,$$
where the fact that $\left(1+\omega^*\right)/2<1$ implies that $\lim_{n\to
	\infty}x_n=0$ and then there is non-reconstruction. So, it
suffices to find some $m$, such that $x_m<\zeta$, which could be
accomplished by choosing $$\varepsilon=
\left(1-\frac{d\theta^2}{\omega^*}\right)\frac\zeta2$$ in
Lemma \ref{Gaussianapproximation}. Then, there exists a sufficiently
large $D=D(\pi, \varepsilon)=D(\pi, \delta)$, such that if $d>D$ and
$x_n\geq\zeta$ then
$$
x_{n+1}\leq g(d\theta^2x_n)+\varepsilon\leq
\frac{d\theta^2}{\omega^*}x_n+\varepsilon\leq
\frac{1}{2}\left(1+\frac{d\theta^2}{\omega^*}\right)x_n.
$$
Then the fact that $\left(1+d\theta^2/\omega^*\right)/2<1$
guarantees the existence of $m$ satisfying $x_m<\zeta$, as desired.
Finally using Proposition 12 in~\cite{mossel2001reconstruction} again, one can conclude
non-reconstruction for any $d\theta^2\leq \omega^*-\delta$.
\end{proof}

\subsection{Proof of Theorem \ref{nonreconstruction}}
\label{Proof_of_Theorem_1.3}
When $1-6\pi_1\pi_2<0$, the proof
of Theorem \ref{nonreconstruction} would resemble
Theorem \ref{reconstruction} in establishing a similar recursive
inequality as equation \eqref{6.4}, under the condition that $x_n\leq\delta$ and
$n\geq N$ for suitable $\delta=\delta(\pi, d)$ and $N=N(\pi)$.
However, there still exists a crucial discrepancy between these two
proofs, that is, Theorem \ref{nonreconstruction} relies heavily on
large $d$. Before we proceed, let us firstly give the following lemma:

\begin{lemma}
	\label{concentrationlarged} For any $0<\varepsilon<1$ and $\alpha
	> 1$, there exist $C=C(\pi, \varepsilon, \alpha)$ and $D=D(\pi, \varepsilon, \alpha)$ such that if
	$d>D$ then 
	\begin{equation}
	\label{lemma6.2}
P\left(\left|\frac{\pi_1Z_1}{\pi_1Z_1+\pi_2Z_2}-\pi_1\right|>\varepsilon\right)\leq Cx_n^\alpha.
	\end{equation}
	Furthermore, there exist $D=D(\pi, \varepsilon)$ and $\delta=\delta(\pi,
	\varepsilon)$ such that if $d>D$ and $x_n\leq\delta$ then
	\begin{equation}
	\label{RxRz} 
	|S|\leq\varepsilon x_n^2.
	\end{equation}
\end{lemma}
\begin{proofsect}{Proof}
	For any $1\leq j\leq d$, define
	\begin{equation*}
	w_j=\theta\frac{Y_j-\pi_1}{\pi_1}-\left(\theta\frac{Y_j-\pi_1}{\pi_1}\right)^2\quad\textup{and}\quad
	M=-\frac{\log(1-\varepsilon)}{4\alpha}.
	\end{equation*}
	From equation \eqref{log} and $|\theta|\leq
	d^{-\frac12}$, one can find a suitable $D=D(\pi, M)>0$ such that when $d\geq D$ we have $\theta<\pi_1$,
	$U_j\geq w_j$ and $|w_j-\mathbf{E}w_j|\leq 2M$. It
	is concluded from Lemma~\ref{covariance} that
	$|d\mathbf{E}w_j|\leq C_1x_n$ and
	\begin{equation}
	\label{Wj-EWj}
\sum_{j=1}^d\mathbf{E}(w_j-\mathbf{E}w_j)^2\leq2d\mathbf{E}\left(\theta\frac{Y_1-\pi_1}{\pi_1}\right)^2+2d\mathbf{E}\left(\theta\frac{Y_1-\pi_1}{\pi_1}\right)^4
\leq4d\mathbf{E}\left(\theta\frac{Y_1-\pi_1}{\pi_1}\right)^2
\leq C_2x_n,
	\end{equation}
	where $C_1$ and $C_2$ denote the constants depending only on $\pi$. In the following context, it is convenient to presume
	\begin{equation}
	\label{xnlog} x_n\leq -\frac{\log(1-\varepsilon)}{2C_1},
	\end{equation}
	for the reason that equation \eqref{lemma6.2} would be trivial otherwise. Therefore, it follows from equation \eqref{xnlog} and the Bennet's inequality that
	\begin{eqnarray*}
	\label{Z1} \mathbf{P}(Z_1\leq
	1-\varepsilon)&=&\mathbf{P}\left(\sum_{j=1}^d U_j\leq
	\log(1-\varepsilon)\right)
	\\
	&\leq&\mathbf{P}\left(\sum_{j=1}^dw_j\leq
	\log(1-\varepsilon)\right)
	\\
	&\leq&\mathbf{P}\left(-\sum_{j=1}^d(w_j-\mathbf{E}w_j)\geq
	-\log(1-\varepsilon)+d\mathbf{E}w_1\right)
	\\
	&\leq&\mathbf{P}\left(\sum_{j=1}^d\left[-(w_j-\mathbf{E}w_j)\right]\geq
	-\frac{\log(1-\varepsilon)}{2}\right)
	\\
	&\leq&\exp\left[\left(-\frac{C_2x_n}{4M^2}\right)\left(-\frac{M\log(1-\varepsilon)}{C_2x_n}\right)\left(\log\frac{-M\log(1-\varepsilon)}{C_2x_n}-1\right)\right] 
	\\
	&\leq&\exp\left[\left(\frac{\log(1-\varepsilon)}{4M}\right)\left(\log\frac{-M\log(1-\varepsilon)}{C_2}-1\right)\right]x_n^{-\frac{\log(1-\varepsilon)}{4M}}
	\\
	&\leq&C_3x_n^\alpha,
	\end{eqnarray*}
	where $C_3$ depends only on $\pi$, $\varepsilon$, $\alpha$. Similarly one can show that $\mathbf{P}(Z_1\geq 1+ \varepsilon) <
	C_4x_n^\alpha$ and then
	\begin{equation*}
	\label{Z2-1} \mathbf{P}(|Z_1-1|>\varepsilon)\leq (C_3+C_4)x_n^\alpha
	\end{equation*}
	for some $C_4=C_4(\pi, \varepsilon, \alpha)$. Similarly one can also show that $\mathbf{P}(|Z_2-1|>\varepsilon)\leq C_5x_n^\alpha$.
On the other hand, there exists $\eta=\eta(\pi, \varepsilon)>0$ such that if
	$|Z_i-1|\leq\eta$ for $i=1, 2$ then 
	$$
\left|\frac{\pi_1Z_1}{\pi_1Z_1+\pi_2Z_2}-\pi_1\right|\leq\varepsilon.
	$$
	Finally, we have
	\begin{equation*}
	\begin{aligned}
	P\left(\left|\frac{\pi_1Z_1}{\pi_1Z_1+\pi_2Z_2}-\pi_1\right|>\varepsilon\right)
	&\leq
	\mathbf{P}\left(\left\{|Z_1-1|>\eta\right\}\cup\left\{|Z_2-1|>\eta\right\}\right)
	\\
	&\leq\mathbf{P}(|Z_1-1|>\eta)+\mathbf{P}(|Z_2-1|>\eta)
	\\
	&\leq Cx_n^\alpha,
	\end{aligned}
	\end{equation*}
	where $C=C(\pi, \varepsilon, \alpha)$. Then we can
	achieve equation \eqref{RxRz} by modifying the proof of Proposition~\ref{estimateforS}.
\end{proofsect}

\begin{lemma}
	\label{7.7} When $1-6\pi_1\pi_2<0$, for any $0< s \leq
	\pi_2$ we have
	$$
	g(s) < s.
	$$
\end{lemma}

\begin{proofsect}{\textbf{Proof of Theorem \ref{nonreconstruction}}}
	Similar to the proof of Theorem \ref{reconstruction}, we will
	analyze $R$ and $S$ in equation \eqref{explicit}
	respectively, under the condition that $1-6\pi_1\pi_2<0$. Taking
	$D_1=C_R^2(6\pi_1\pi_2^2)^2/(6\pi_1\pi_2-1)^2$, if
	$d>D_1$ which implies $|\theta|\leq d^{-1/2}\leq D_1^{-1/2}$, by equation \eqref{R} and the inequality that
	$\left|z_n/x_n-\pi_1\right|\leq1$, we obtain
	\begin{eqnarray}
	\label{7.1R}  
	C_R\frac{d(d-1)}{2}|\theta|^5\left|\frac{z_n}{x_n}-\pi_1\right|x_n^2
	\leq\frac{1}{6}\frac{6\pi_1\pi_2-1}{\pi_1\pi_2^2}\frac{d(d-1)}{2}\theta^4x_n^2.
	\end{eqnarray}
	
Moreover, according to Lemma~\ref{concentrationlarged}, there exist
	$D_2=D_2(\pi)>D_1$ and $\delta=\delta(\pi)>0$ independent of $d$, such that
	if $d>D_2$ and $x_n<\delta$ then an analogue of equation \eqref{6.2} holds as
	\begin{equation}
	\label{7.2S}
	|S|\leq\frac{1}{4}\frac{6\pi_1\pi_2-1}{\pi_1\pi_2^2}\frac{d(d-1)}{2}\theta^4x_n^2,
	\end{equation}
	and then by equation \eqref{7.1R} we have
	\begin{equation}
	\label{7.3}
	|R|\leq\frac{1}{4}\frac{6\pi_1\pi_2-1}{\pi_1\pi_2^2}\frac{d(d-1)}{2}\theta^4x_n^2.
	\end{equation}
	
Next we claim that there is a positive integer $m$ such that $x_m<\delta$.
	Define $\varepsilon=\varepsilon(\pi,
	\delta)=\varepsilon(\pi)=\frac{1}{2}\min_{s\geq \delta}(s-g(s))$.
	Since the function $s-g(s)$ is continuous and positive on $[\delta, \pi_2]$ by Lemma~\ref{7.7}, we have $\varepsilon>0$. Then by
	Lemma \ref{Gaussianapproximation}, there exists a $D=D(\pi,
	\varepsilon)=D(\pi)>D_2>0$, such that when $d>D$,
	$$
	|x_{n+1}-g(d\theta^2x_n)|<\varepsilon,
	$$
    thus if $x_n\geq\delta$ then
	$$
	x_{n+1} < g(d\theta^2x_n) +\varepsilon
	\leq g(x_n)+\varepsilon
	\leq x_n-2\varepsilon+\varepsilon
	\leq x_n-\varepsilon,
	$$
	where the second inequality follows from Lemma \ref{increasing} which claims that
	$g(s)$ is increasing on $[0, \pi_2]$. Therefore, there must exist an
	$m\in \mathbb{Z}^+$, such that $x_m<\delta$, as desired.

	When $d>D$, it can be shown by induction, equation \eqref{7.2S} and equation \eqref{7.3} that when $n\geq m$, 
	\begin{equation}
	\label{7.2} x_{n+1}\leq
	d\theta^2x_n-\frac{1}{2}\frac{(6\pi_1\pi_2-1)}{\pi_1\pi_2^2}\frac{d(d-1)}{2}\theta^4x_n^2
	\leq d\theta^2x_n
	\leq x_n< \delta.
	\end{equation}
	Therefore, the limit $L$ defined as $L=\lim_{n\to\infty}x_n\geq0$ does exist, since
	the sequence $\{x_n\}_{n\geq m}$ is bounded and decreasing. Thus,
	passing to the limit on both sides of equation \eqref{7.2} gives
	$$
	L\leq
	d\theta^2L-\frac{1}{2}\frac{(6\pi_1\pi_2-1)}{\pi_1\pi_2^2}\frac{d(d-1)}{2}\theta^4L^2,
	$$
	which implies $L=0$, hence
	non-reconstruction.

\end{proofsect}

\section*{Acknowledgement}
We give special thanks to the journal editor
and two anonymous reviewers who provided us with many constructive and helpful comments.
We also give special thanks to Sebastien Roch for his
inspiring discussions and reading of some proofs in the first version of this paper.  We truly appreciate the
warm encouragements on finishing this complete version and helpful discussions on this topic as well as future extensions, from colleagues at the 2017 \& 2018 Columbia-Princeton Probability Day, 2017 Northeast Probability Seminar, 2018 Frontier Probability Days, 2017 \& 2018 Finger Lakes Probability Seminar, and 2017 \& 2018 Seminar on Stochastic Processes. 
\bibliography{\jobname}

\end{document}